\documentclass[]{article} 
\usepackage{graphicx}
\usepackage{amsfonts}
\usepackage[fleqn]{amsmath}
\usepackage{amssymb}
\usepackage{amsthm}
\usepackage{amscd}
\usepackage[T1]{fontenc}
\usepackage{afterpage}  %
\usepackage{fancyhdr}
\theoremstyle{plain}
\newtheorem{theorem}{Theorem}[section]
\newtheorem{lemma}[theorem]{Lemma}
\newtheorem{proposition}[theorem]{Proposition}

\newtheorem{corollary}[theorem]{Corollary}
\theoremstyle{definition}
\newtheorem{definition}[theorem]{Definition}
\newtheorem{example}[theorem]{Example}
\theoremstyle{remark}

\begin{document}

\afterpage{\rhead[]{\thepage} \chead[\small W. A. Dudek and R. A. R. Monzo]
{\small Translatable quadratical quasigroups} \lhead[\thepage]{} }                  

\begin{center}
\vspace*{2pt}
{\Large \textbf{Translatable quadratical quasigroups}}\\[30pt]
 {\large \textsf{\emph{Wieslaw A. Dudek$^*$ \ and \ Robert A. R. Monzo}}}\\[30pt]
\end{center}
 {\footnotesize\textbf{Abstract.} 
The concept of a $k$-translatable groupoid is introduced. Those $k$-translatable quadratical quasigroups induced by the additive group of integers modulo $m$, where $k<40$, are listed for $m\leqslant 1200$. The fine structure of quadratical quasigroups is explored in detail and the Cayley tables of quadratical quasigroups of orders $5$, $9$, $13$ and $17$ are produced. All but those of order $9$ are $k$-translatable, for some $k$. Quadratical quasigroups induced by the additive group of integers modulo $m$ are proved to be $k$-translatable, for some $k$. Open questions and thoughts about future research in this area are given.
 }
\footnote{\textsf{2010 Mathematics Subject Classification:} 20M15, 20N02}
\footnote{\textsf{Keywords:} quadratical quasigroup, $k$-translatable groupoid, cycle.}

\section*{\centerline{1. Introduction}}\setcounter{section}{1}

Geometrical motivations for the study of quadratical quasigroups have been given in \cite{Vol1, Vol2,Vol3, Vol4}. In particular Volenec \cite{Vol1,Vol2} defined a product $*$ on $\mathbb{C}$, the complex numbers, that defines a quadratical quasigroup. The product $x*y$ of two distinct elements is the third vertex of a positively oriented, isosceles right triangle in the complex plane, at which the right angle occurs.

The main aim of this paper is to give insight into the fine algebraic structure of quadratical quasigroups, in order to set the stage for, and to stimulate, further development of the general theory that is still in its relative infancy. This is the second of a series of four papers that advance this theory. We concern ourselves here mainly with the fine algebraic structure, rather than with the geometrical representations, of quadratical quasigroups. However, as noted by Volenec, each algebraic identity valid in the quadratical quasigroup $(\mathbb{C},*)$ can be interpreted as a geometrical theorem and the theory of quadratical quasigroups gives a better insight into the mutual relations of such theorems (\cite{Vol1}, page 108).

Volenec \cite{Vol1} proved that quadratical quasigroups have a number of properties, such as idempotency, mediality and cancellativity. These properties were applied by the authors in \cite{DM} to prove that quadratical quasigroups form a variety $\mathcal{Q}$. The spectrum of $\mathcal{Q}$ was proved to be contained in the set of all integers equal to $1$ plus a multiple of $4$.  Quadratical quasigroups are uniquely determined by certain abelian groups and their automorphisms \cite{D1}. Necessary and sufficient conditions under which $\mathbb{Z}_m$, the additive group of integers modulo $m$, induces quadratical quasigroups are given in \cite{DM}.
 
This paper builds on the authors' work in \cite{DM}, as well as the prior work of Polonijo \cite{Pol}, Volenec \cite{Vol1} and Dudek \cite{D1}. In Sections $3$, $4$, $5$, $6$ and $7$ the notion of a {\it four-cycle}, which was introduced in \cite{DM}, is used to explore in detail the fine structure of quadratical quasigroups. The concept of a four-cycle is applied in Sections $4$ and $6$ to produce Cayley tables for quadratical quasigroups of orders $5$, $9$, $13$ and $17$. These tables can be reproduced by model builders, but we would not achieve our aim of stimulating thought about the fine algebraic structure in that manner.

In Section $8$, all of these quadratical quasigroups except those of order $9$ are proved to be $k$-translatable, for some $k$. We prove that, up to isomorphism, there is only one quadratical quasigroup of order $9$ and that it is self-dual. Quadratical quasigroups of order $25$ and $29$ are found. The one of order $25$ is $18$-translatable, its dual is $7$-translatable, the quadratical quasigroup of order $29$ is $12$-translatable and its dual is $17$-translatable.

Sections $8$ and $9$ of this paper explore other ways of constructing $k$-translatable quasigroups. We introduce the central concept of a $k$-translatable groupoid in Section $8$ and use it to characterize quadratical quasigroups. In Section $9$ necessary and sufficient conditions are found for a quasigroup induced by $\mathbb{Z}_m$ to be $k$-translatable. We prove that a quadratical quasigroup induced by $\mathbb{Z}_m$ is always $k$-translatable, for some $k$. The existence of $k$-translatable quadratical quasigroups induced by some $\mathbb{Z}_m$ is established for each integer $k$, where $1<k<11$. Values of $m$ for which a quadratical quasigroup induced by $\mathbb{Z}_m$ is $(m-k)$-translatable are determined for each integer $k$, where $1<k<11$.

In Section $9$ lists are given for $k$-translatable $(k<40)$ quadratical quasigroups of orders $m <1200$, induced by $\mathbb{Z}_m$ and $k$-translatable quadratical quasigroups induced by $\mathbb{Z}_m$ for $m<500$. 

In a future publication, the two different approaches to the construction of quadratical quasigroups are united. It will be proved that a quadratical quasigroup is translatable if and only if it is induced by some $\mathbb{Z}_{4n+1}$. Finally, open questions and possible future directions for research are discussed in Section 9.

\section*{\centerline{2. Preliminaries}}\setcounter{section}{2}
Volenec \cite{Vol1} defined a {\it quadratical groupoid} as a right solvable groupoid satisfying the following condition:
$$
\rule{45mm}{0mm}xy \cdot x = zx \cdot yz.\rule{45mm}{0mm}(A)
$$ 

He proved that such groupoids are quasigroups and satisfy the identities listed below.

\begin{theorem}\label{T2.1} A quadratical groupoid satisfies the following identities:
\begin{eqnarray}
&&x = {x^2} \ \ \  (idempotency),\label{e1}\\
&&x \cdot yx = xy \cdot x \ \ \ (elasticity),\label{e2}\\
&&x \cdot yx = xy \cdot x=yx \cdot y \ \ \ (strong \ elasticity),\label{e3}\\
&&yx \cdot xy = x \ \ \ (bookend),\label{e4}\\
&&x \cdot yz = xy \cdot xz \ \ \ (left \ distributivity),\label{e5}\\
&&xy \cdot z = xz \cdot yz \ \ \ (right \ distributivity),\label{e6}\\
&&xy \cdot zw = xz \cdot yw \ \ \ (mediality),\label{e7}\\
&&x(y \cdot yx) = (xy \cdot x)y, \ \ \\\ \label{e8}
&&(xy \cdot y)x = y(x \cdot yx), \\ \label{e9}
&&xy = zw\,\longleftrightarrow\,yz = wx \ \ \ (alterability).\label{e10}
\end{eqnarray}
\end{theorem}

These identities can be used to characterize quadratical quasigroups. Namely, the following theorem is proved in \cite{DM}. 

\begin{theorem}\label{T2.2}
The class of all quadratical quasigroups form a variety uniquely defined by

$\bullet$ \ $(A)$, $\eqref{e3}$, $\eqref{e4}$, $\eqref{e7}$, or

$\bullet$ \ $\eqref{e1}$, $\eqref{e4}$, $\eqref{e7}$, or

$\bullet$ \ $\eqref{e2}$, $\eqref{e4}$, $\eqref{e7}$, or

$\bullet$ \ $\eqref{e4}$, $\eqref{e5}$, $\eqref{e10}$.
\end{theorem}

Quadratical quasigroups are uniquely characterized by commutative groups and their automorphisms. This characterization (proved in \cite{D1}) is presented below.

\begin{theorem}\label{T-gr}
A groupoid $(G,\cdot)$ is a quadratical quasigroup if and only if there exists a commutative group $(G,+)$ in which for every $a\in G$ the equation $z+z=a$ has a unique solution $z\in G$ and $\varphi,\psi$ are automorphisms of $(G,+)$ such that 
$$
xy=\varphi(x)+\psi(y),
$$
$$
\varphi(x)+\psi(x)=x,
$$
$$
2\varphi\psi(x)=x
$$
for all $x,y\in G$.
 \end{theorem}  

In this case we say that the quadratical quasigroup is {\em induced by $(G,+)$}.

We also will need the following two results proved in \cite{DM}.

\begin{theorem}\label{T2.3}
A finite quadratical groupoid has order $m=4t+1$.
\end{theorem}
So, later it will be assumed that $m=4t+1$ for some natural $t$.

\begin{theorem}\label{T2.4} A quadratical groupoid induced by the additive group $\mathbb{Z}_m$ has the form
$$x\cdot y=ax+(1-a)y\, ,
$$
where $\;a\in\mathbb{Z}_m$ and
\begin{equation}\label{ea}
\rule{35mm}{0mm}2a^2-2a+1=0 .
\end{equation}
\end{theorem}

\section*{\centerline{3. Products in quadratical quasigroups}}\setcounter{section}{3}\setcounter{theorem}{0}

Let $Q$ be a quadratical quasigroup and $a,b\in Q$ be two different elements. Suppose that $C=\{x_1,x_2,\ldots,x_n\}\subseteq Q$ consists of $n$ distinct elements, such that $aba=x_1x_2=x_2x_3=x_3x_4=\ldots=x_{n-1}x_n=x_nx_1$. Then $C$ will be called an {\it $($ordered$)$ $n$-cycle based on} $aba$. Note that $x_1\ne aba$, or else $x_1=x_2=\ldots=x_n=aba$. Note also that if $C=\{x_1,x_2,x_3,\ldots,x_n\}\subseteq Q$ is an $n$-cycle based on $aba$, then so is  $C_i=\{x_i,x_{(i+1)mod\,n},x_{(i+2)mod\,n},\ldots,x_{(i+n-1)mod\,n}\}$. 

In \cite{DM} is proved that in a quadratical quasigroup all $n$-cycles have the length $n=4$. Moreover, if $a,b\in Q$ and $a\ne b$, then each element $x_1\ne aba$ of $Q$ is a member of a $4$-cycle based on $aba$. Two $4$-cycles based on $aba$, where $a\ne b$, are equal or disjoint. Note that in any $4$-cycle $C=\{x_1,x_2,x_3,x_4\}$, $x_4=x_1x_3$. Hence, $C=\{x,yx,y,xy\}$, where $x=x_1$ and $y=x_3$.
 
\begin{definition}\label{D3.1} Let $Q$ be a quadratical quasigroup with $\{a,b\}\subseteq Q$ and $a\ne b$. Then $\{a,b,ab,ba,aba\}$ contains five distinct elements. We will use the notation $[1,1]=a$, $[1,2]=ab$, $[1,3]=ba$ and $[1,4]=b$. We omit the commas and square brackets in the notation, when this causes no confusion, and write $11=a$, $12=ab$, $13=ba$ and $14=b$. For $n\geqslant 2$, by induction we define $n1=(n-1)1\cdot (n-1)2$, $n2 = (n-1)2\cdot (n-1)4$, $n3 = (n-1)3\cdot (n-1)1$,
$n4 = (n-1)4\cdot (n-1)3$ and $Hn = \{n1, n2, n3, n4\}$. On the occasions when we need to highlight that the element $fk$, $f\in\{1,2,\ldots,n\}$ and $k\in\{1,2,3,4\}$, is in the dual quadratical quasigroup $Q^*$ we will denote it by $fk^*$. Similarly, $Hn^* = \{n1^*, n2^*, n3^*, n4^*\}$. Note that the values of both $fk$ and $fk^*$ depend on the choice of the elements $a$ and $b$. 
\end{definition}
\begin{example}\label{Ex3.2}  $H2=\{a\cdot ab,ab\cdot b,ba\cdot a,b\cdot ba\}$,\\
$H3=\{(a\cdot ab)(ab\cdot b),(ab\cdot b)(b\cdot ba),(ba\cdot a)(a\cdot ab),(b\cdot ba)(ba\cdot a)\}$,\\
$H4=\{(31\cdot 32)(32\cdot 34), (32\cdot 34)(34\cdot 33),(33\cdot 31)(31\cdot 32), (34\cdot 33)(33\cdot 31)\}$, where\\
 $31=(a\cdot ab)(ab\cdot b)$, $32=(ab\cdot b)(b\cdot ba)$, $33= (ba\cdot a)(a\cdot ab)$ and $34=(b\cdot ba)(ba\cdot a)$. 
\end{example}
\begin{example}\label{Ex3.3} $11^* = a$, $12^* =a*b$, $13^* =b*a$, $14^* =b$ and, for $n\geqslant 2$, by induction we define $n1^* = (n-1)1^**(n-1)2^*$, $n2^* = (n-1)2^**(n-1)4^*$, $n3^* = (n-1)3^**(n-1)1^*$ and $n4^* = (n-1)4^* *(n-1)3^*$.
\end{example}
\begin{example}\label{Ex3.4}  $H2^*\! =\!\{a*(a*b), (a*b)*b,(b*a)*a,b*(b*a)\} = \{ba\cdot a, b\cdot ba, a\cdot ab, ab\cdot b\}$ and 
$52^*=42^*\cdot 44^*=(32^*\cdot 34^*)(34^*\cdot 33^*)=  
( ((ab*b)* (b*ba))*((b*ba)*(ba*a)))*(((b* ba)*  (ba* a)) *(ba*a) * (a* ab))$, where $
a* ab =a*(a*b)$, $ab*b =(a*b)*b$, $ba*a =(b*a)*a$  and $b*ba = b*(b*a)$.
\end{example}
Note that the expression $ab$, when working in the dual groupoid $Q^* = (Q,*)$, equals $a*b$, which equals $b\cdot a$ in the original groupoid itself. This notation will cause no problems, as we will either calculate values only using the dot product or the star product, or when we are calculating using both products, as in Theorem \ref{T5.1}, the distinction will be obvious.

The proofs of the following propositions are straightforward, using induction on $n$ and the properties of quadratical quasigroups, and are omitted.

\begin{proposition}\label{P3.5} For any positive integer $t$, \ $t1\cdot t4=t2$, $t2\cdot t3=t4$, $t3\cdot t2=t1$ and $t4\cdot t1=t3$.
\end{proposition}
\begin{proposition}\label{P3.6} For $t>1$, \ $aba\cdot tk=(t-1)k$ \ for any $k\in\{1,2,3,4\}$.
\end{proposition}
\begin{proposition}\label{P3.7} For $t>1$, \ $t1\cdot aba=(t-1)2$, $t2\cdot aba=(t-1)4$, $t3\cdot aba=(t-1)1$ and $t4\cdot aba=(t-1)3$.
\end{proposition}
\begin{proposition}\label{P3.8} For any positive integer $t$,\ $Ht$ contains $4$ distinct elements.
\end{proposition}
\begin{proposition}\label{P3.9} For any positive integer $t$,\ $Ht\cap \{aba\}=\emptyset$.
\end{proposition}
\begin{proposition}\label{P3.10} For any positive integer $t$, \ $t1\cdot t3=t2\cdot t1=t3\cdot t4=t4\cdot t2=aba$.
\end{proposition}
\begin{proposition}\label{P3.11} $Ht=\{t1,t3,t4,t2\}$ is a $4$-cycle based on $aba$. 
\end{proposition}
\begin{definition}\label{D3.12} We say that a {\it groupoid $Q$ is of the form} $Qn$, for some positive integer $n$, if $\;Q=\{aba\}\!\bigcup\limits_{t=1}^n\!\!H t\;$ for some $\{a,b\}\subseteq Q$, where each $Ht$ is as in Definition \ref{D3.1}.
\end{definition}

\section*{\centerline{4. Quadratical quasigroups of form Q1 and Q2 }}\setcounter{section}{4}
\setcounter{theorem}{0}

We are now in a position to examine more closely the Cayley tables of quadra\-tical quasigroups. This will aid in the construction of the tables for quadratical quasigroups of orders $5$, $9$, $13$ and $17$. Dudek \cite{D1} gave two examples of quadratical quasigroups of orders $5$, $13$ and $17$ and six examples of quadratical quasigroups of order $9$. A close examination of the fine structure will aid us in proving that all these quadratical quasigroups are of the form $Qn$, for some positive integer $n$. Each pair of quadratical quasigroups of orders $5$, $13$ or $17$ will be proved to be dual groupoids. The $6$ quadratical quasigroups of order $9$ will be proved to be of form $Q2$ and self-dual. That is, up to isomorphism, there is only one quadratical quasigroup of order $9$.

A method of constructing quadratical quasigroups of the form $Qn$ is as follows. Proposition \ref{P3.6} implies that $aba \cdot Ht =H(t-1)$ for all $t \ne 1$. Since quadratical quasigroups are cancellative, we can assume that $aba \cdot H1 =Hn$. If we choose the value of $aba \cdot 11$ in $Hn = \{n1,n2,n3,n4\}$ then, using the properties of quadratical quasigroups, we can attempt to fill in the remaining unknown products in the Cayley table. If this can be done without contradiction, then, using Theorem \ref{T2.2}, we can check that the groupoid thus obtained is quadratical, by
checking that it is bookend and medial. Completing the Cayley table is this way is not always possible, as shown in the following example.

\begin{example}\label{Ex4.1} Suppose $Q$ is a quadratical quasigroup of the form $Q2$. Then $aba \cdot 11 = aba \cdot a \in H2 = \{ 21,22,23,24\} = \{a \cdot ab,ab \cdot b,ba \cdot a,b \cdot ba\}$. Now $aba \cdot a = a(ba \cdot a)$ and so $aba \cdot a \notin \{a \cdot ab,ba \cdot a\}$, since cancellativity, idempotency and alterability would imply that $a = b$ (if $aba \cdot a = ba \cdot a$) and $b = a \cdot ab$ (if $aba \cdot a = a \cdot ab$), the latter contradicting to the fact that two $4$-cycles based on $aba$ are  equal or disjoint (cf. \cite{DM}). Hence, $aba \cdot a$ must be in the set $\{ab \cdot b,b \cdot ba\}$. However, if $aba \cdot a = b \cdot ba$, then by $(10)$, $ab = ba \cdot aba = (b \cdot ab)a = aba \cdot a = b \cdot ba$, a contradiction since $H1 \cap H2 = \emptyset $.
\end{example}

Example \ref{Ex4.1} shows that $aba \cdot a = ab \cdot b$. Using the properties of quadratical quasigroups, the Cayley table of the groupoid of the form $Q2$ can only be completed in one way, as shown below here, in Table 1. 

We then need to calculate all the possible products $xy \cdot yx$ and $xy \cdot zw$ in Table 1, to prove that they are equal to $y$ and $xz \cdot yw$ respectively. Then, by Theorem \ref{T2.2}, $Q2$ would be quadratical. This proves to be the case and we omit the detailed calculations. However, to give a flavour of the calculations we find all products $aba \cdot x$ and $x\cdot aba$ when $x\in H1$ and $aba \cdot a = ab \cdot b$.

Since $( a \cdot aba)\left( {aba \cdot a} \right) = \left( {a \cdot aba} \right)\left( {ab \cdot b} \right)$, it follows that we have $a\cdot aba = b \cdot ba$, $aba \cdot b = ba \cdot a$,
$aba \cdot ab = \left( {aba \cdot a} \right)\left({aba \cdot b} \right) = \left( {ab \cdot b} \right)\left( {ba \cdot a} \right) = b \cdot ba$ and, si\-mi\-larly $aba \cdot ba = a \cdot ab$. Then $aba \cdot ab = b \cdot ba$ implies
$ba \cdot aba = ab \cdot b$. Also, $aba = \left( {ab \cdot aba} \right)\left( {aba \cdot ab} \right) = \left( {ab \cdot aba} \right)(b \cdot ba)$ implies $ab \cdot aba = ba \cdot a$. Finally,
$b \cdot aba = \left( {ab \cdot aba} \right)\left( {ba \cdot aba} \right) = $ $\left( {ba \cdot a} \right)\left( {ab \cdot b} \right) = a \cdot ab$. \\
                                                         
\noindent
{\small{
\begin{tabular}{|l|c|c|c|c|c|c|c|c|c|}\hline
$\rule{0mm}{3mm}\;\;\;Q2$&$\!\!\!11\!=\!a\!\!\!$&$\!\!\!12\!=\!ab\!\!\!$&$\!\!\!13\!=\!ba\!\!\!$&$\!\!\!\!\!14\!=\!b\!\!\!$\!\!&\!\!$\!\!\!\!\!aba\!\!\!\!\!\!\!$&$\!\!\!21\!=\!a\!\cdot\! ab\!\!\!$&$\!\!\!22\!=\!ab\!\cdot\! b\!\!\!$&$\!\!\!23\!=\!ba\!\cdot\! a\!\!\!$&$\!\!\!24\!=\!b\!\cdot\! ba\!\!\!$\\ \hline
$\!\!11\!=\!a$&$a$&$a\cdot ab$&$aba$&$ab$&$b\cdot ba$&$ba$&$b$&$ab\cdot b$&$ba\cdot a$\\ \hline
$\!\!12\!=\!ab$&$aba$&$ab$&$b$&\!\!\!$\!ab\cdot b\!$&$ba\cdot a$&$b\cdot ba$&$a$&$a\cdot ab$&$ba$\\ \hline
$\!\!13\!=\!ba$&$ba\cdot a$&$a$&$ba$&$aba$&\!$\!ab\cdot b\!$&$ab$&$b\cdot ba$&$b$&$a\cdot ab$\\ \hline
$\!\!14\!=\!b$&$ba$&$aba$&$b\cdot ba$&$b$&\!$\!a\cdot ab\!$&$ab\cdot b$&$ba\cdot a$&$a$&$ab$\\ \hline
$\;\;aba$&$ab\cdot b$&$b\cdot ba$&$a\cdot ab$&\!\!$ba\cdot a$&$aba$&$a$&$ab$&$ba$&$b$\\ \hline
$\!\!21\!=\!a\!\cdot\! ab\!\!\!\!$\!\!&$b\cdot ba$&$b$&$ba\cdot a$&$a$&$ab$&$a\cdot ab$&$ba$&$aba$&$ab\cdot b$\\ \hline
$\!\!22\!=\!ab\!\cdot\! b\!\!\!$\!\!&$a\cdot ab$&$ba\cdot a$&$ab$&$ba$&$b$&$aba$&$ab\cdot b$&$b\cdot ba$&$a$\\ \hline
$\!\!23\!=\!ba\!\cdot\! a\!\!\!$&$ab$&$ba$&$ab\cdot b$&\!\!$b\cdot ba$&$a$&$b$&$a\cdot ab$&$ba\cdot a$&$aba$\\ \hline
$\!\!24\!=\!b\!\cdot\! ba$\!\!\!&$b$&$ab\cdot b$&$a$&\!\!$a\cdot ab$&$ba$&$ba\cdot a$&$aba$&$ab$&$b\cdot ba$\\ \hline
 \end{tabular}}}

\medskip
\centerline{\small Table 1.}

\begin{proposition} 
A quadratical quasigroup $Q$ of order $9$ is of the form $Q=Q2$.
\end{proposition}
\begin{proof} We have $Q=H1\cup\{aba\}\cup C$, where $C$ is a $4$-cycle based on $aba$ and $C\cap H1=\emptyset$. We proceed to prove that $C=H2$.

Consider the following part of the Cayley table: $(H1\cup\{aba\})\cdot H1$. 

\medskip\centerline{
{\small{
\begin{tabular}{|c|c|c|c|c|c|c|c|c|}\hline
$\rule{0mm}{3mm}\;Q$&$\!\!\!a\!\!\!$&$\!\!\!ab\!\!\!$&$\!\!\!ba\!\!\!$&$\!\!\!b\!\!\!$\\ \hline
$a$&$a$&&$aba$&$ab$\\ \hline
$ab$&$aba$&$ab$&$b$&\\ \hline
$ba$&&$a$&$ba$&$aba$\\ \hline
$b$&$ba$&$aba$&&$b$\\ \hline
$aba$&&&&\\ \hline
 \end{tabular}}}}

\medskip

From the table, clearly, if $ba\cdot a\in H1\cup\{aba\}$, then $ba\cdot a\in\{ab,b\}$.

Assume that $ba\cdot a=ab$. Then we have $a=b\cdot ba$, \ $ab\cdot b=(ba\cdot a)b=(ba\cdot b)\cdot ab=aba\cdot ab=a(ba\cdot b)=(b\cdot ba)(ba\cdot b)=b(ba\cdot ab)=ba$  and $b=a\cdot ab$. Also, $aba\cdot a=a(ba\cdot a)=a\cdot ab=b$, \ $aba\cdot ab=ab\cdot (a\cdot ab)=ab\cdot b=ba$, \ $aba\cdot b=bab\cdot b=b(ab\cdot b)=b\cdot ba=a$ and $aba\cdot ba=(aba\cdot b)(aba\cdot a)=ab$. So, we have proved that $(H1\cup\{aba\})\cdot H1=H1\cup\{aba\}$.

Similarly, if  $ba\cdot b=b$, then $(H1\cup\{aba\})\cdot H1=H1\cup\{aba\}$, which is not possible because, if $c\in C$, then $c\in C=\{ca,c\cdot ab,c\cdot ba,cb\}$, a contradiction. So, $ba\cdot a=c$, for some $c\in C$. Then, since $C=\{c,dc,d,cd\}$ for some $d\in C$, we have $aba=c\cdot dc=dc\cdot d=d\cdot cd=cd\cdot c$. So, $aba=(ba\cdot a)\cdot dc$, which implies $dc=b\cdot ba$. Also, $aba=cd\cdot (ba\cdot a)$, which implies $cd=a\cdot ab$. Then, $aba=dc\cdot d=(b\cdot ba)d$, which gives $d=ab\cdot b$. Hence, $C=\{a\cdot ab,ab\cdot b,ba\cdot   a, b\cdot ba\}=H2$.  
\end{proof}
So, we have proved that a quadratical quasigroup of order 9 must be the quasigroup Q2.

\medskip\noindent
{\bf Open question.} {\it Is a finite, idempotent, alterable, cancellative, elastic groupoid of  form $Qn$ quadra\-tical?} 

Note that we can prove that the answer is affirmative when $n=1$ or $n=2$.

\bigskip
Now, if we calculate the Cayley table for $(Q2)^*$, the dual of $Q2$, we see that the table for the dual product $*$ (defined as $a * b = b \cdot a$) is exactly the same as Table $1$, where the product is the dot product $\cdot $. (For example, $(( b*a)*a)*(b*a)=(a\cdot ab)*ab=ab\cdot (a\cdot ab) =b\cdot ba =(a*b)*b$ and, by Table $1$,
$(ba \cdot a)\cdot ba = ab\cdot b$). Hence, $Q2 \cong (Q2)^*$. Another way to put this is that the quadratical groupoid $Q2$ must be self-dual. An isomorphism $\theta $ between $Q2$ and $(Q2)^*$ is: $\theta a = a$, $\theta b = b$, $\theta (ab) = a * b$,
$\theta (ba) = b * a$, $\theta (a \cdot ab) = a * (a * b)$, $\theta (ab\cdot b) = (a*b)*b$, $\theta (ba \cdot a) = (b*a)*a\,$ and $\,\theta (b \cdot ba) = b*(b*a)$. 

\begin{example}\label{Ex4.2} It is straightforward to calculate the Cayley tables of the quadra\-tical quasigroups, each of order $9$, given in \cite{D1}. They are each based on the group $\mathbb{Z}_3\times\mathbb{Z}_3$ of ordered pairs of integers, with product being addition (mod 3). The products are defined as follows:

\smallskip
$\left( {x,y} \right) *_1 \left( {z,u} \right) = \left( {y + z + 2u,{\rm{ }}x + y + 2z} \right)$,

$\left( {x,y} \right) *_2 \left( {z,u} \right) = \left( {2y + z + u,{\rm{ 2}}x + y{\rm{ + }}z} \right),$ 

$\left( {x,y} \right) *_3 \left( {z,u} \right) = \left( {x + y + 2u,{\rm{ }}x + 2z + u} \right)$,

$\left( {x,y} \right) *_4 \left( {z,u} \right) = \left( {x + 2y + u,{\rm{ 2}}x + z + u} \right),$

$(x,y) *_5 \left( {z,u} \right) = \left( {2x + y + 2z + 2u,{\rm{ 2}}x + 2y + z + 2u} \right)$,

$(x,y) *_6 \left( {z,u} \right) = \left( {2x + 2y + 2z + u,{\rm{ }}x + 2y{\rm{ + 2}}z + 2u} \right).$

\medskip
In each table, if we calculate $ab$ and $ba$ for the ordered pairs $a = (1,1)$ and $b = (1,2)$ we see that $Q=\{aba\}\cup H1\cup H2$ and that $aba\cdot a = ab\cdot b$. Therefore, these six quadratical quasigroups are isomorphic to each other and to $Q2$. We already knew that there is only one quadratical quasigroup of order $9$, but these calculations clarify (and reinforce a conviction) that the quadratical quasigroups of order $9$ presented in \cite{D1} are isomorphic. 
\end{example}

\begin{example}\label{Ex4.3} We now calculate the Cayley table for a groupoid $Q1$ and its dual, when $aba\cdot a \in\{ab,b\}$. \\ 
                                                                         
																																
\noindent
{\small{\begin{tabular}{lcr}
\begin{tabular}{|c|c|c|c|c|c|}\hline
$\rule{0mm}{3mm}Q1$\!\!&$a$&$ab$&$ba$&$b$&$aba$\\ \hline
$a$&$a$&$ba$&$aba$\!&$ab$&$b$\\ \hline
$ab$&$aba$\!&$ab$&$b$&$a$&$ba$\\ \hline
$ba$&$b$&$a$&$ba$&$aba$\!&$ab$\\ \hline
$b$&$	ba$&$aba$\!&$ab$&$b$&$a$\\ \hline
$aba$\!&$ab$&$b$&$a$&$ba$&$aba$\!\\ \hline
\end{tabular}
& \ \ \ &
\begin{tabular}{|c|c|c|c|c|c|}\hline
$\rule{0mm}{3mm}(Q1)^*$\!\!\!\!&$a$&\!\!$b*a$\!\!&\!\!$a*b$\!\!&$b$&\!\!$aba$\!\!\\ \hline
$a$&$a$&$aba$&$b$&\!\!$a*b$\!\!&\!\!\!\!$b*a$\!\!\!\!\\ \hline
\!\!$b*a$\!\!&$a*b$&\!\!$b*a$\!\!&$a$&\!\!$aba$\!\!&$b$\\ \hline
\!\!$a*b$\!\!&\!\!$aba$\!\!&$b$&\!\!$a*b$\!\!&$b*a$&$a$\\ \hline
$b$&\!\!$b*a$\!\!&$a$&$aba$&$b$&\!\!\!\!$a*b$\!\!\!\!\\ \hline
$aba$&$b$&\!\!$a*b$\!\!&\!\!$b*a$\!\!&$a$&$aba$\\ \hline
\end{tabular}
\end{tabular}}}

\medskip
\centerline{\small Table 2.}

\medskip

Checking these tables shows that each is medial and bookend and that, indeed, these two quadratical quasigroups are dual.
\end{example}

\medskip\noindent
{\bf Open question.} {\it Examining Tables $1$ and $2$ closely, we can show that any two distinct elements of $Q1$ $($resp. $(Q1)^*$, $Q2)$ generate $Q1$ $($resp. $(Q1)^*$, $Q2)$. This will later be seen to be the case also for $Q3$, $Q4$ and their duals. We conjecture that if $Q$ is a quadratical quasigroup of form $Qn$, for some positive integer $n$, then it is gene\-rated by any two distinct elements. Such a property does not hold in quadratical quasigroups in general, as we shall now prove.

\noindent
\begin{example}\label{Ex4.4} Since $Q$ is a variety of groupoids, the direct product of quadratical quasigroups is quadratical. Hence, $Q1\times Q1$ is quadratical. If we choose a base element, $(a,b)$ say, then $Q1\times Q1$ consists of six disjoint $4$-cycles based on $(a,b)$; namely, 

\smallskip
$\{(a,a),(a,aba),(a,ab),(a,ba)\}$, \ \ \ \ \ \,$\{(b,ab),(aba,ba),(ba,a),(ab,aba)\}$,

\smallskip
$\{(ab,b),(b,b),(aba,b),(ba,b)\}$, \ \ \ \ \ \ \ $\{(ab,ab),(b,ba),(aba,a),(ba,aba)\}$, 

\smallskip
$\{(ba,ba),(ab,a),(b,aba),(aba,ab)\}$, \ \ $\{(aba,aba),(ba,ab),(ab,ba),(b,a)\}$. 

\medskip
If $C$ is any one of these six $4$-cycles, then no two distinct elements $x$ and $y$ of $C$ generates $Q1\times Q1$, because $\{x,y\}\subseteq C$ and $C$ is a proper subquadratical quasigroup of $Q1\times Q1$, isomorphic to $Q1$.
\end{example}
\begin{example}\label{Ex4.5} $(Q1\times Q1)^*=(Q1)^*\times (Q1)^*$ and $\left(Q1\times (Q1)^*\right)^*= (Q1)^*\times Q1$.  Note that $(a,ba)$ and $(ab,b)$ generate $Q1\times (Q1)^*$ and $(ba,a)$ and $(b,ab)$ generate $(Q1)^*\times Q1$ while $Q1\times Q1$ and $(Q1)^*\times (Q1)^*$ are not $2$-generated.
\end{example}

\section*{\centerline{5. The elements nk*}}\setcounter{section}{5}
\setcounter{theorem}{0}

\rm
The following Theorem is easily proved for $k=1$  and, by induction on $k$, is straightforward to prove for all $k\in\{0,1,2,\ldots\}=\mathbb{N}_0$. The proof is omitted.

\begin{theorem}\label{T5.1} For all $k\in \mathbb{N}_0$,

\smallskip

\hspace*{-8mm}
{\footnotesize $ 
\begin{array}{lllllll}
(\!(1\!+\!4k)1)^*\!\!=\!(1\!+\!4k)1,&\! (\!(1\!+\!4k)2)^*\!\!=\!(1\!+\!4k)3,&\! (\!(1\!+\!4k)3)^*\!\!=\!(1\!+\!4k)2,&\!(\!(1\!+\!4k)4)^*\!\!=\!(1\!+\!4k)4,\\
(\!(2\!+\!4k)1)^*\!\!=\!(2\!+\!4k)3,&\! (\!(2\!+\!4k)2)^*\!\!=\!(2\!+\!4k)4,&\! (\!(2\!+\!4k)3)^*\!\!=\!(2\!+\!4k)1,&\!(\!(2\!+\!4k)4)^*\!\!=\!(2\!+\!4k)2,\\
(\!(3\!+\!4k)1)^*\!\!=\!(3\!+\!4k)4,&\! (\!(3\!+\!4k)2)^*\!\!=\!(3\!+\!4k)2,&\! (\!(3\!+\!4k)3)^*\!\!=\!(3\!+\!4k)3,&\!(\!(3\!+\!4k)4)^*\!\!=\!(3\!+\!4k)1,\\
(\!(4\!+\!4k)1)^*\!\!=\!(4\!+\!4k)2,&\! (\!(4\!+\!4k)2)^*\!\!=\!(4\!+\!4k)1,&\! (\!(4\!+\!4k)3)^*\!\!=\!(4\!+\!4k)4,&\!(\!(4\!+\!4k)4)^*\!\!=\!(4\!+\!4k)3.
\end{array} 
 $}
 \end{theorem}

\smallskip

Further, for simplicity, elements of the form $(xy)^*$ will be denoted as $xy^*$.

\smallskip

Now, considering the quadratical quasigroups of form $Qn$, from the remarks in the paragraph preceding Example \ref{Ex4.1}, we see that there are at most $4$ groupoids of the form $Qn$ for any given integer $n$. Since the dual of a quadratical quasigroup of the form $Qn$ must also have the form $Qn$, we can tell, from the following Theorem, which values of $aba\cdot a$ may yield groupoids that are duals of each other.  

\begin{theorem}\label{T5.2} For all positive integers $n\geqslant 2$, the following identities are valid in a quadratical quasigroup of form $Qn$, depending on the value of \ $aba\cdot a$:\\

 \noindent\hspace*{-2mm}
{\small{
\begin{tabular}{|c|c|c|c|c|c|c|c|c|c|c|c|c|c|c|c|}\hline
\!\!$\rule{0mm}{3mm}aba\!\cdot\!a$\!\!&\!\!$aba\!\cdot\!ab$\!\!&\!\!$aba\!\cdot\!ba$\!\!&\!\!$aba\!\cdot\!b$\!\!&\!\!$a\!\cdot\!aba$\!\!&\!\!$ab\!\cdot\!aba$\!\!&\!\!$ba\!\cdot\!aba$\!\!&\!\!$b\!\cdot\!aba$\!\!&\!\!$n1\!\cdot\!n2$\!\!&\!\!$n2\!\cdot\!n4$\!\!&\!\!$n3\!\cdot\!n1$\!\!&\!\!$n4\!\cdot\!n3$\!\!\\ \hline

$n1$&$n2$&$n3$&$n4$&$n2$&$n4$&$n1$&$n3$&$a$&$ab$&$ba$&$b$\\ \hline

$n2$&$n4$&$n1$&$n3$&$n4$&$n3$&$n2$&$n1$&$ba$&$a$&$b$&$ab$\\ \hline

$n3$&$n1$&$n4$&$n2$&$n1$&$n2$&$n3$&$n4$&$ab$&$b$&$a$&$ba$\\ \hline
$n4$&$n3$&$n2$&$n1$&$n3$&$n1$&$n4$&$n2$&$b$&$ba$&$ab$&$a$\\ \hline
\end{tabular}}}

 \bigskip
  \hspace{10mm}{\small{
\begin{tabular}{|c|c|c|c|c|c|c|c|c|c|c|c|c|c|c|c|}\hline
\!\!$aba\!\cdot\!a$\!\!&\!\!$11\!\cdot\!34$\!\!&\!\!$23\!\cdot\!14$\!\!&\!\!$34\!\cdot\!14$\!\!&\!~\!$14\!\cdot\!21$\!\!&\\ \hline

$n1$&$n3$&$n2$&$n1$&$n1$&$(n\!-\!1)2=11\!\cdot\! n1=n2\!\cdot\! 11$\\ \hline

$n2$&$n1$&$n4$&$n2$&$n2$&$(n\!-\!1)4=11\!\cdot\! n2=n4\!\cdot\! 11$\\ \hline

$n3$&$n4$&$n1$&$n3$&$n3$&$(n\!-\!1)1=11\cdot\! n3=n1\!\cdot\! 11$\\ \hline
$n4$&$n2$&$n3$&$n4$&$n4$&$(n\!-\!1)3=11\!\cdot\! n4=n3\!\cdot\! 11$\\ \hline
\end{tabular}}}
\end{theorem}
\smallskip
\begin{proof} We prove only the identities for when $aba\cdot a=n2$, as the proofs of the other three cases are similar. We have $aba\cdot n2$. Then, $aba=(a\cdot aba)(aba\cdot a)=(a\cdot aba)\cdot n2$. By Proposition \ref{P3.11} and Theorem \ref{T2.1}, $a\cdot aba=n4=a\cdot bab=aba\cdot ab$. Then, $n4=aba\cdot ab=(aba\cdot a)(aba\cdot b)=n2\cdot (aba\cdot b)$. By Proposition \ref{P3.5}, $aba\cdot b=n3$. So $aba\cdot ba=(aba\cdot b)(aba\cdot a)=n3\cdot n2=n1$  (by Proposition \ref{P3.5}). Then, $aba=(b\cdot aba)(aba\cdot b)=(b\cdot aba)\cdot n3$, which by Proposition \ref{P3.11} implies $b\cdot aba=n1$. Then, using Proposition \ref{P3.5}, $ab\cdot aba=(a\cdot aba)(b\cdot aba)=n4\cdot n1=n3$ and $ba\cdot aba=(b\cdot aba)(a\cdot aba)=n1\cdot n4=n2$. We also have $n1\cdot n2=(aba\cdot ba)(ba\cdot aba)=ba$, $n2\cdot n4=(aba\cdot a)(a\cdot aba))=a$, \ $n3\cdot n1=(aba\cdot b)(b\cdot aba)=b$ and $n4\cdot n3=(aba\cdot ab)(ab\cdot aba)=ab$. Now, $11\cdot 34=a\cdot (b\cdot ba)(ba\cdot a)=a(b\cdot ba)\cdot a(ba\cdot a)=(ab\cdot aba)(aba\cdot a)=n3\cdot n2=n1$, \ $34\cdot 14=(b\cdot ba)(ba\cdot a)\cdot b=(b\cdot ba)b\cdot (ba\cdot a)b=(b\cdot bab)(bab\cdot ab)=(b\cdot aba)(aba\cdot ab)=n1\cdot n4=n2$, \ $14\cdot 21=b(a\cdot ab)=ba\cdot bab=ba\cdot aba=n2$ \ and \ $23\cdot 14=(ba\cdot a)b=bab\cdot ab=aba\cdot ab=n4$.       

Finally, $a\cdot n2=a\cdot aba\cdot a=aba\cdot a\cdot aba=aba\cdot n4=(n-1)4=11\cdot n2$ and
$n4\cdot a=a\cdot aba\cdot a=aba\cdot a\cdot aba=(n-1)4=n4\cdot 11$. 

This completes the proof of the validity of the identities indicated in row $3$ of the two tables in Theorem \ref{T5.2}, when $aba\cdot a=n2$. 
\end{proof} 

As mentioned above, Theorem \ref{T5.2} will be useful when we look for the duals of the quadratical quasigroups that we will call $Q3$ and $Q4$, as will the following concept.

\begin{definition}\label{D5.3} If a quadratical quasigroup of form $Qn$ exists for some integer $n$ then the identity generated on the left (on the right) by an identity $kr\cdot ls=mt$, where $r,s,t\in\{1,2,3,4\}$ and $k,l,m\leqslant n$, is defined as the identity 
$$(aba\cdot kr)(aba\cdot ls)=aba\cdot mt  \ \ \ \ ({\rm resp.}\; (kr\cdot aba)(ls\cdot aba)=mt\cdot aba)
$$
 and $kr\cdot ls=mt$ is called the {\it generating identity}. 
\end{definition}

Note that Propositions \ref{P3.6} and \ref{P3.7}, along with Theorem \ref{T5.2}, give the means of calculating identities generated on the left and right by a given identity. Multiplying on the left (or on the right) repeatedly $n$-times gives $n$ distinct identities. These methods will later be used to prove that quadratical quasigroups of the form $Q6$ do not exist. 

\section*{\centerline{6. Quadratical quasigroups of forms Q3 and Q4}}\setcounter{section}{6}
\setcounter{theorem}{0}

We give the Cayley tables of quadratical quasigroups of orders $13$ and $17$. 

First we note that for a quadratical quasigroup of form $Q3$, if $aba\cdot a=n3=33=(ba\cdot a)(a\cdot ab)$, then $aba\cdot a=a(ba\cdot a)=(a\cdot ab)\cdot aba=ab$, which implies, by cancellation, $ba\cdot a=b$, a contradiction because $H1\cap H2=\emptyset$. If $aba\cdot a=n4=34=(b\cdot ba)(ba\cdot a)$, then $ab\cdot aba=a(b\cdot  ba)=(ba\cdot a)\cdot aba=a$, which implies $b\cdot ba=a$, a contradiction. Hence, $aba\cdot a\in\{31,\,32\}=\{(a\cdot ab)(ab\cdot b),\,(ab\cdot b)(b\cdot ba)\}$. Setting $aba\cdot a=a\cdot ab$ and using the properties of quadratical quasigroups (Theorem \ref{T2.1}) we obtain the Cayley Table $3$. It can be checked that it is medial and bookend and so, by Theorem \ref{T2.2}, this groupoid is a quadratical quasigroup.\\    

\medskip
{\small $
\begin{array}{|c|c|c|c|c|c|c|c|c|c|c|c|c|c|c|c|}\hline\rule{0mm}{3mm}
Q3&11&12&13&14&aba&21&22&23&24&31&32&33&34\\ \hline
11&11&21&aba&12&32&14&23&31&34&22&13&24&33\\ \hline
12&aba&12&14&22&34&32&13&33&21&23&24&31&11\\ \hline
13&	23&11&13&aba&31&24&32&12&33&14&34&21&22\\ \hline
14&13&aba&24&14&33&31&34&22&11&32&21&12&23\\ \hline
aba&31&32&33&34&aba&11&12&13&14&21&22&23&24\\ \hline
21&32&23&34&13&12&21&31&aba&22&24&33&11&14\\ \hline
22&33&34&11&21&14&aba&22&24&32&12&23&13&31\\ \hline
23&24&14&31&32&11&33&21&23&aba&34&12&22&13\\ \hline
24&12&31&22&33&13&23&aba&34&24&11&14&32&21\\ \hline
31&34&13&21&24&22&12&33&14&23&31&11&aba&32\\ \hline
32&22&33&23&11&24&13&14&21&31&aba&32&34&12\\ \hline
33&14&22&32&23&21&34&24&11&12&13&31&33&aba\\ \hline
34&21&24&12&31&23&22&11&32&13&33&aba&14&34\\ \hline
\end{array}$}

\medskip
\centerline{\small Table 3.}

\medskip

There are then two ways to obtain the Cayley table for $(Q3)^*$. Firstly, we can use $aba*a=32^*=[(a*b)*b]*[b*(b*a)]$ and, using the properties of quadratical quasigroups, we can then calculate the remaining products in Table $4$. 

Alternatively, we can calculate the products directly from Table $3$, using our Theorem \ref{T5.1}. For example, $23^*=(b*a)*a=a\cdot ab=21$, and similarly $32^*=((a*b)*b)*(b*(b*a))=(ab\cdot b)(b\cdot ba)=32$. Hence, $32^**23^*=32*21=21\cdot 32$. From Table $3$, $21\cdot 32=33$. But from Theorem \ref{T5.1}, $33 = 33^*$. So, we obtain
$32^**23^*=33 =33^*$. The remaining products in Table $4$ can be calculated in similar fashion. Having already checked that Table $3$ is quadratical, Table $4$ also produces a quadratical quasigroup, the dual groupoid.

\medskip

\hspace*{-4mm}
{\small $
\begin{array}{|c|c|c|c|c|c|c|c|c|c|c|c|c|c|c|c|}\hline\rule{0mm}{3mm}
(Q3)^*&11^*&12^*&13^*&14^*&aba&21^*&22^*&23^*&24^*&31^*&32^*&33^*&34^*\\ \hline
11^*&	11^*&	21^*&	aba&12^*&	34^*&	22^*&	13^*&	32^*&	33^*&	23^*&	24^*&	14^*&	31^*\\ \hline
12^*&aba&12^*&	14^*&	22^*&	33^*&	34^*&	24^*&	31^*&	11^*&	13^*&	21^*&	32^*&	23^*\\ \hline
13^*&23^*&	11^*&	13^*&	aba&32^*&	14^*&	34^*&	21^*&	31^*&	22^*&	33^*&	24^*&	12^*\\ \hline
14^*&	13^*&	aba	&24^*&	14^*&	31^*&	32^*&	33^*&	12^*&23^*&34^*&	11^*&	21^*&	22^*\\ \hline
aba&32^*&	34^*&	31^*&	33^*&	aba&11^*&	12^*&	13^*&	14^*&	21^*&	22^*&	23^*&	24^*\\ \hline
21^*&	34^*&	13^*&	33^*&	24^*&	12^*&	21^*&	31^*&	aba&22^*&	32^*&	23^*&	11^*&	14^*\\ \hline
22^*&	31^*&	33^*&	23^*&	11^*&	14^*&	aba&22^*&	24^*&	32^*&	12^*&	34^*&	13^*&	21^*\\ \hline
23^*&	14^*&	22^*&	32^*&	34^*&	11^*&	33^*&	21^*&	23^*&	aba&23^*&	12^*&	31^*&	13^*\\ \hline
24^*&	21^*&	32^*&	12^*&	31^*&	13^*&	23^*&	aba&34^*&	24^*&	11^*&	14^*&	22^*&	33^*\\ \hline
31^*&	33^*&	24^*&	11^*&	21^*&	22^*&	12^*&	23^*&	14^*&	34^*&	31^*&	13^*&	aba&	32^*\\ \hline
32^*&	12^*&	31^*&	22^*&	23^*&	24^*&	13^*&	14^*&	33^*&	21^*&	aba&32^*&	34^*&	11^*\\ \hline
33^*&	22^*&	23^*&	34^*&	13^*&	21^*&	24^*&	32^*&	11^*&	12^*&	14^*&	31^*&	33^*&	aba\\ \hline
34^*&	24^*&	14^*&	21^*&	32^*&	23^*&	31^*&	11^*&	22^*&	13^*&	33^*&	aba&12^*&	34^*\\ \hline
\end{array}$}

\medskip
\centerline{\small Table 4.}

\medskip

Similarly, we can calculate the Cayley tables for $Q4$ and its dual $(Q4)^*$:

 \hspace*{-4mm}
{\small $
\begin{array}{|c|c|c|c|c|c|c|c|c|c|c|c|c|c|c|c|c|c|c|c|c|}\hline\rule{0mm}{3mm}
Q4&11&12&13&14&\!\!aba\!\!&21&22&23&24&31&32&33&34&41&42&43&44\\ \hline
11&	11&	21&	\!\!aba\!\!&	12&	44&	24&	32&	42&	43&	14&	23&	31&	41&	33&	34&	13&	22\\ \hline
12&	\!\!aba\!\!&12&	14&	22&	43&	44&	23&	41&	34&	32&	13&	42&	21&	11&	31&	24&	33\\ \hline
13&	23&	11&	13&\!\!aba\!\!&	42&	31&	44&	22&	41&	24&	43&	12&	33&	32&	21&	34&	14\\ \hline
14&	13&	\!\!aba\!\!&24&	14&	41&	42&	43&	33&	21&	44&	34&	22&	11&	23&	12&	31&	32\\ \hline
\!\!aba\!\!&42&	44&	41&	43&	\!\!aba\!\!&11&	12&	13&	14&	21&	22&	23&	24&	31&	32&	33&	34\\ \hline
21&	44&	32&	43&	23&	12&	21&	31&\!\!aba\!\!&	22&	34&	42&	11&	14&	24&	33&	41&	13\\ \hline
22&	41&	43&	21&	34&	14&	\!\!aba\!\!&22&	24&	32&	12&	33&	13&	44&	42&	23&	11&	31\\ \hline
23&	31&	24&	42&	44&	11&	33&	21&	23&	\!\!aba\!\!&41&	12&	32&	13&	34&	14&	22&	43\\ \hline
24&	22&	42&	33&	41&	13&	23&\!\!aba\!\!&	34&	24&	11&	14&	43&	31&	12&	44&	32&	21\\ \hline
31&	43&	23&	34&	13&	22&	12&	42&	14&	33&	31&	41&\!\!aba\!\!&	32&	44&	11&	21&	24\\ \hline
32&	33&	41&	11&	21&	24&	13&	14&	31&	44&\!\!aba\!\!&	32&	34&	42&	22&	43&	23&	12\\ \hline
33&24&	14&	44&	32&	21&	41&	34&	11&	12&	43&	31&	33&	\!\!aba\!\!&13&	22&	42&	23\\ \hline
34&	12&	31&	22&	42&	23&	32&	11&	43&	13&	33&	\!\!aba\!\!&44&	34&	21&	24&	14&	41\\ \hline
41&	21&	34&	12&	31&	32&	14&	33&	44&	23&	22&	11&	24&	43&	41&	13&	\!\!aba\!\!&	42\\ \hline
42&	14&	22&	32&	33&34&	43&	13&	21&	31&	23&	24&	41&	12&\!\!aba\!\!&	42&	44&	11\\ \hline
43&	32&	33&	23&	11&31&	34&	24&	12&	42&	13&	44&	21&	22&	14&	41&	43&	\!\!aba\!\!\\ \hline
44&	34&	13&	31&	24&	33&	22&	41&	32&	11&	42&	21&	14&	23&	43&\!\!aba\!\!&	12&	44\\ \hline
\end{array}$}

\medskip
\centerline{\small Table 5.}

\bigskip
  
\hspace*{-6mm}
{\small $
\begin{array}{|c|c|c|c|c|c|c|c|c|c|c|c|c|c|c|c|c|c|c|}\hline\rule{0mm}{3mm}
\!\!(Q4)^*\!\!\!&\!\!11^*\!\!\!&\!\!	12^*\!\!\!&\!\!	13^*\!\!\!&\!\!	14^*\!\!\!&\!\!	\!\!\!aba\!\!\!\!\!&\!\!	21^*\!\!\!&\!\!	22^*\!\!\!&\!\!	23^*\!\!\!&\!\!	24^*\!\!\!&\!\!	31^*\!\!\!&\!\!	32^*\!\!\!&\!\!	33^*\!\!\!&\!\!	34^*\!\!\!&\!\!	41^*\!\!\!&\!\!	42^*\!\!\!&\!\!	43^*\!\!\!&\!\!	44^*\!\!\\ \hline
11^*\!\!&\!	11^*\!\!&\!	21^*\!\!&\!	\!\!\!aba\!\!\!\!&\!	12^*\!\!&\!	41^*\!\!&\!	34^*\!\!&\!	24^*\!\!&\!	43^*\!\!&\!	42^*\!\!&\!	13^*\!\!&\!	33^*\!\!&\!	22^*\!\!&\!	44^*\!\!&\!	14^*\!\!&\!	23^*\!\!&\!	31^*\!\!&\!	32^*\!\!\\ \hline
12^*\!\!&\!	\!\!\!aba\!\!\!\!&\!	12^*\!\!&\!	14^*\!\!&\!	22^*\!\!&\!	42^*\!\!&\!	41^*\!\!&\!	33^*\!\!&\!	44^*\!\!&\!	23^*\!\!&\!	24^*\!\!&\!	11^*\!\!&\!	43^*\!\!&\!	31^*\!\!&\!	32^*\!\!&\!	13^*\!\!&\!	34^*\!\!&\!	21^*\!\!\\ \hline
13^*\!\!\!&\!\!	23^*\!\!\!&\!\!	11^*\!\!\!&\!\!	13^*\!\!\!&\!\!	\!\!\!aba\!\!\!\!\!&\!\!	43^*\!\!\!&\!\!	22^*\!\!\!&\!\!	41^*\!\!\!&\!\!	32^*\!\!\!&\!\!	44^*\!\!\!&\!\!	34^*\!\!\!&\!\!	42^*\!\!\!&\!\!	14^*\!\!\!&\!\!	21^*\!\!\!&\!\!	24^*\!\!\!&\!\!	31^*\!\!\!&\!\!	12^*\!\!\!&\!\!	33^*\!\!\\ \hline
14^*\!\!\!&\!\!	13^*\!\!\!&\!\!	\!\!\!aba\!\!\!\!\!&\!\!	24^*\!\!\!&\!\!	14^*\!\!\!&\!\!	44^*\!\!\!&\!\!	43^*\!\!\!&\!\!	42^*\!\!\!&\!\!	21^*\!\!\!&\!\!	31^*\!\!\!&\!\!	41^*\!\!\!&\!\!	23^*\!\!\!&\!\!	32^*\!\!\!&\!\!	12^*\!\!\!&\!\!	33^*\!\!\!&\!\!	34^*\!\!\!&\!\!	22^*\!\!\!&\!\!	11^*\!\!\\ \hline
\!\!\!aba\!\!\!\!\!&\!\!43^*\!\!\!&\!\!	41^*\!\!\!&\!\!	44^*\!\!\!&\!\!	42^*\!\!\!&\!\!	\!\!\!aba\!\!\!\!\!&\!\!	11^*\!\!\!&\!\!	12^*\!\!\!&\!\!	13^*\!\!\!&\!\!	14^*\!\!\!&\!\!	21^*\!\!\!&\!\!	22^*\!\!\!&\!\!	23^*\!\!\!&\!\!	24^*\!\!\!&\!\!	31^*\!\!\!&\!\!	32^*\!\!\!&\!\!	33^*\!\!\!&\!\!	34^*\!\!\\ \hline
21^*\!\!\!&\!\!	41^*\!\!\!&\!\!	24^*\!\!\!&\!\!	42^*\!\!\!&\!\!	33^*\!\!\!&\!\!	12^*\!\!\!&\!\!	21^*\!\!\!&\!\!	31^*\!\!\!&\!\!	\!\!\!aba\!\!\!\!\!&\!\!	22^*\!\!\!&\!\!	44^*\!\!\!&\!\!	34^*\!\!\!&\!\!	11^*\!\!\!&\!\!	14^*\!\!\!&\!\!	23^*\!\!\!&\!\!	43^*\!\!\!&\!\!	32^*\!\!\!&\!\!	13^*\!\!\\ \hline
22^*\!\!\!&\!\!	44^*\!\!\!&\!\!	42^*\!\!\!&\!\!	31^*\!\!\!&\!\!	23^*\!\!\!&\!\!	14^*\!\!\!&\!\!	\!\!\!aba\!\!\!\!\!&\!\!	22^*\!\!\!&\!\!	24^*\!\!\!&\!\!	32^*\!\!\!&\!\!	12^*\!\!\!&\!\!	43^*\!\!\!&\!\!	13^*\!\!\!&\!\!	33^*\!\!\!&\!\!	34^*\!\!\!&\!\!	21^*\!\!\!&\!\!	11^*\!\!\!&\!\!	41^*\!\!\\ \hline
23^*\!\!\!&\!\!	22^*\!\!\!&\!\!	34^*\!\!\!&\!\!	43^*\!\!\!&\!\!	41^*\!\!\!&\!\!	11^*\!\!\!&\!\!	33^*\!\!\!&\!\!	21^*\!\!\!&\!\!	23^*\!\!\!&\!\!	\!\!\!aba\!\!\!\!\!&\!\!	32^*\!\!\!&\!\!	12^*\!\!\!&\!\!	42^*\!\!\!&\!\!	13^*\!\!\!&\!\!	44^*\!\!\!&\!\!	14^*\!\!\!&\!\!	24^*\!\!\!&\!\!	31^*\!\!\\ \hline
24^*\!\!\!&\!\!	32^*\!\!\!&\!\!	43^*\!\!\!&\!\!	21^*\!\!\!&\!\!	44^*\!\!\!&\!\!	13^*\!\!\!&\!\!	23^*\!\!\!&\!\!	\!\!\!aba\!\!\!\!\!&\!\!	34^*\!\!\!&\!\!	24^*\!\!\!&\!\!	11^*\!\!\!&\!\!	14^*\!\!\!&\!\!	31^*\!\!\!&\!\!	41^*\!\!\!&\!\!	12^*\!\!\!&\!\!	33^*\!\!\!&\!\!	42^*\!\!\!&\!\!	22^*\!\!\\ \hline
31^*\!\!\!&\!\!	42^*\!\!\!&\!\!	33^*\!\!\!&\!\!	23^*\!\!\!&\!\!	11^*\!\!\!&\!\!	22^*\!\!\!&\!\!	12^*\!\!\!&\!\!	34^*\!\!\!&\!\!	14^*\!\!\!&\!\!	43^*\!\!\!&\!\!	31^*\!\!\!&\!\!	41^*\!\!\!&\!\!	\!\!\!aba\!\!\!\!\!&\!\!	32^*\!\!\!&\!\!	13^*\!\!\!&\!\!	44^*\!\!\!&\!\!	21^*\!\!\!&\!\!	24^*\!\!\\ \hline
32^*\!\!\!&\!\!	21^*\!\!\!&\!\!	44^*\!\!\!&\!\!	12^*\!\!\!&\!\!	31^*\!\!\!&\!\!	24^*\!\!\!&\!\!	13^*\!\!\!&\!\!	14^*\!\!\!&\!\!	41^*\!\!\!&\!\!	33^*\!\!\!&\!\!	\!\!\!aba\!\!\!\!\!&\!\!	32^*\!\!\!&\!\!	34^*\!\!\!&\!\!	42^*\!\!\!&\!\!	22^*\!\!\!&\!\!	11^*\!\!\!&\!\!	23^*\!\!\!&\!\!	43^*\!\!\\ \hline
33^*\!\!\!&\!\!	34^*\!\!\!&\!\!	13^*\!\!\!&\!\!	41^*\!\!\!&\!\!	24^*\!\!\!&\!\!	21^*\!\!\!&\!\!	32^*\!\!\!&\!\!	44^*\!\!\!&\!\!	11^*\!\!\!&\!\!	12^*\!\!\!&\!\!	43^*\!\!\!&\!\!	31^*\!\!\!&\!\!	33^*\!\!\!&\!\!	\!\!\!aba\!\!\!\!\!&\!\!	42^*\!\!\!&\!\!	22^*\!\!\!&\!\!	14^*\!\!\!&\!\!	23^*\!\!\\ \hline
34^*\!\!\!&\!\!	14^*\!\!\!&\!\!	22^*\!\!\!&\!\!	32^*\!\!\!&\!\!	43^*\!\!\!&\!\!	23^*\!\!\!&\!\!	42^*\!\!\!&\!\!	11^*\!\!\!&\!\!	31^*\!\!\!&\!\!	13^*\!\!\!&\!\!	33^*\!\!\!&\!\!	\!\!\!aba\!\!\!\!\!&\!\!44^*\!\!\!&\!\!	34^*\!\!\!&\!\!	21^*\!\!\!&\!\!	24^*\!\!\!&\!\!	41^*\!\!\!&\!\!	12^*\!\!\\ \hline
41^*\!\!\!&\!\!	31^*\!\!\!&\!\!	23^*\!\!\!&\!\!	34^*\!\!\!&\!\!	13^*\!\!\!&\!\!	32^*\!\!\!&\!\!	14^*\!\!\!&\!\!	43^*\!\!\!&\!\!	33^*\!\!\!&\!\!	21^*\!\!\!&\!\!	22^*\!\!\!&\!\!	44^*\!\!\!&\!\!	24^*\!\!\!&\!\!	11^*\!\!\!&\!\!	41^*\!\!\!&\!\!	12^*\!\!\!&\!\!	\!\!\!aba\!\!\!\!\!&\!\!	42^*\!\!\\ \hline
42^*\!\!\!&\!\!	33^*\!\!\!&\!\!	32^*\!\!\!&\!\!	11^*\!\!\!&\!\!	21^*\!\!\!&\!\!	34^*\!\!\!&\!\!	31^*\!\!\!&\!\!	13^*\!\!\!&\!\!	22^*\!\!\!&\!\!	41^*\!\!\!&\!\!	23^*\!\!\!&\!\!	24^*\!\!\!&\!\!	12^*\!\!\!&\!\!	43^*\!\!\!&\!\!	\!\!\!aba\!\!\!\!\!&\!\!	42^*\!\!\!&\!\!	44^*\!\!\!&\!\!	14^*\!\!\\ \hline
43^*\!\!\!&\!\!	24^*\!\!\!&\!\!	14^*\!\!\!&\!\!	33^*\!\!\!&\!\!	32^*\!\!\!&\!\!	31^*\!\!\!&\!\!	44^*\!\!\!&\!\!	23^*\!\!\!&\!\!	12^*\!\!\!&\!\!	34^*\!\!\!&\!\!	42^*\!\!\!&\!\!	13^*\!\!\!&\!\!	21^*\!\!\!&\!\!	22^*\!\!\!&\!\!	11^*\!\!\!&\!\!	41^*\!\!\!&\!\!	43^*\!\!\!&\!\!	\!\!\!aba\!\!\!\!\!\\ \hline
44^*\!\!\!&\!\!	12^*\!\!\!&\!\!	31^*\!\!\!&\!\!	22^*\!\!\!&\!\!	34^*\!\!\!&\!\!	33^*\!\!\!&\!\!	24^*\!\!\!&\!\!	32^*\!\!\!&\!\!	42^*\!\!\!&\!\!	11^*\!\!\!&\!\!	14^*\!\!\!&\!\!	21^*\!\!\!&\!\!	41^*\!\!\!&\!\!	23^*\!\!\!&\!\!	43^*\!\!\!&\!\!	\!\!\!aba\!\!\!\!\!&\!\!	13^*\!\!\!&\!\!	44^*\!\!\\ \hline
\end{array}$}

\medskip
\centerline{\small Table 6.}

\medskip
Groups of orders $13$ and $17$ are isomorphic to the additive groups $\mathbb{Z}_{13}$ and $\mathbb{Z}_{17}$, respectively. So, by Theorem \ref{T2.4}, quasigroups $Q3$ and $Q4$ are isomorphic to quadratical quasigroups induced by $\mathbb{Z}_{13}$ and $\mathbb{Z}_{17}$, respectively. Direct computations show that $Q3$ is isomorphic to the quadratical quasigroup $(\mathbb{Z}_{13},\cdot)$ with the operation $x\cdot y=11x+3y({\rm mod}\,13)$; the dual quasigroup $(Q3)^*$ is isomorphic to the quasigroup $(\mathbb{Z}_{13},\circ)$ with the operation $x\circ y= 3x+11y({\rm mod}\,13)$.
Similarly, $Q4$ is isomorphic to $(\mathbb{Z}_{17},\cdot)$ with the operation $x\cdot y=11x+7y({\rm mod}\,17)$. Its dual quasigroup $(Q4)^*$ is isomorphic to the quasigroup $(\mathbb{Z}_{17},\circ)$ with the operation $x\circ y= 7x+11y({\rm mod}\,17)$.

\section*{\centerline{7. No quadratical quasigroup of form Q6 exists}}\setcounter{section}{7}
\setcounter{theorem}{0}

\begin{theorem}\label{T7.1}  There is no quadratical quasigroup of form $Q6$.
\end{theorem}
\begin{proof}  {\sc Case 1}: $aba\cdot a=61$. Using Propositions \ref{P3.6}, \ref{P3.7}, Theorem \ref{T5.2} and Theorem \ref{T2.1}, we see that $aba\cdot a=61$, by \eqref{e10}, implies
\begin{equation}\label{e12} 
a\cdot 61=61\cdot aba=52=62\cdot 11\stackrel{\eqref{e10}}{=}aba\cdot 62=52\cdot 52. 
\end{equation}
Then,  $62=61\cdot 64$, by Proposition \ref{P3.5}. This, by Proposition \ref{P3.6}, gives $52=51\cdot 54$. Also, $62=52\cdot 54$, by Definition \ref{D3.1}, whence $52=42\cdot 44$, by Proposition \ref{P3.6} and \eqref{e10}, and so
\begin{equation}\label{e13} 
52=51\cdot 54=42\cdot 44. 
\end{equation}
Theorem \ref{T5.2} implies $61=14\cdot 21=34\cdot14$, \  $62=23\cdot 14$ and $63=11\cdot 34$. So, these identities generate the following:
\begin{equation}\label{e14} 
52=63\cdot 12=13\cdot 64=64\cdot 21=23\cdot 63. 
\end{equation}

As a consequence of \eqref{e12}, \eqref{e13}, \eqref{e14}, Proposition \ref{P3.5} and Proposition \ref{P3.10} we can see that the solutions to the equation $52=12\cdot x$ must be in the set $\{14,22,23,24,31,32,33,34,41,42,43,51,53\}.$ Now, by Definition \ref{D3.1}, we obtain  $22=12\cdot 14\ne 52$ and so $x\ne 14$.
		
To eliminate the other possibilities for $x$ we now use the generating identities $(15)$ through $(25)$, indicated in the Table $7$ below. 

\bigskip
	\hspace*{-7mm}
{\small $
\begin{array}{|c|c|c|c|c|c|c|c|c|c|c|c|c|c|c|c|c|c|c|}\hline\rule{0mm}{3mm}
&\!(15)&\!	(16)&\!	(17)&\!	(18)&\!	(19)	&\!(20)&\!	(21)&\!	(22)&\!	(23)&\!	(24)&\!	(25)\!\\ \hline
&\!\!(n\!\!-\!\!1)2\!\!&\!\!(n\!\!-\!\!1)2\!\!&\!aba\!\cdot\!\!  11&\!11\!\!\cdot\!  aba&\!\!Prop.&\!Def.&\!n1=&\!n2=&\!n3=&\!n1=&\!\!idem.\!\\
&\!\!=\!\!11\!\!\cdot\!  n1\!\!&\!\!=\!n2\!\cdot\!\!  11\!\!&\!=61&\!=62&\!\ref{P3.5}&\!\ref{D3.1}&\!14\!\cdot\!  21&\!23\!\cdot\!  14&\!11\!\cdot\!  34&\!34\!\cdot\!  14&\!\!\\ \hline											
\!\!52\!\!&\!11\!\cdot\!  61&\!62\!\cdot\!  11&\!
aba\!\cdot\!   62\!&\!
61\!\cdot\!   aba\!&\!
51\!\cdot\!   54\!&\!
42\!\cdot\!   44\!&\!
63\!\cdot\!   12\!&\!
13\!\cdot\!   64\!&\!
64\!\cdot\!   21\!&\!
23\!\cdot\!   13\!&\!
52\!\cdot\!   52\!\!\\ \hline

\!\!44\!&\!	62\!\cdot\!   52&\!
54\!\cdot\!   62&\!
aba\!\cdot\!   54\!&\!
52\!\cdot\!   aba\!&\!
42\!\cdot\!   43\!&\!
34\!\cdot\!  33\!&\!
51\!\cdot\!   64\!&\!
61\!\cdot\!  22\!&\!
53\!\cdot\!   12\!&\!
11\!\cdot\!  61\!&\!
44\!\cdot\!  44\!\!\\ \hline

\!\!33\!\!&\!	54\!\cdot\!  44\!&\!
43\!\cdot\!  54\!&\!
aba\!\cdot\!  43\!&\!
44\!\cdot\!  aba\!&\!
34\!\cdot\! 31\!&\!
23\!\cdot\!  21\!&\!
42\!\cdot\! 53\!&\!
52\!\cdot\!  41\!&\!
41\!\cdot\!  64\!&\!
62\!\cdot\!  52\!&\!
33\!\cdot\!  33\!\!\\ \hline

\!\!21\!\!&\!	43\!\cdot\!  33\!&\!
31\!\cdot\!  43\!&\!
aba\!\cdot\!  31\!&\!
33\!\cdot\!  aba\!&\!
23\!\cdot\!  22\!&\!
11\!\cdot\!  12\!&\!
34\!\cdot\!  41\!&\!
44\!\cdot\!  32\!&\!
32\!\cdot\!  53\!&\!
54\!\cdot\!  34\!&\!
21\!\cdot\!  21\!\!\\ \hline

\!\!12\!\!&\!	31\!\cdot\!  21\!&\!
22\!\cdot\!  31\!&\!
aba\!\cdot\!  22\!&\!
21\!\cdot\!  aba\!&\!
11\!\cdot\!  14\!&\!
62\!\cdot\!  64\!&\!
23\!\cdot\!  32\!&\!
33\!\cdot\!  24\!&\!
24\!\cdot\!  41\!&\!
43\!\cdot\!  33\!&\!
12\!\cdot\!  12\!\!\\ \hline
  
\!\!64\!\!&\!	22\!\cdot\!  12\!&\!
14\!\cdot\!  22\!&\!
aba\!\cdot\!  14\!&\!
12\!\cdot\!  aba\!&\!
62\!\cdot\!  63\!&\!
54\!\cdot\!  53\!&\!
11\!\cdot\!  24\!&\!
21\!\cdot\!  13\!&\!
13\!\cdot\!  32\!&\!
31\!\cdot\!  21\!&\!
64\!\cdot\!  64\!\!\\ \hline

\!\!53\!\!&\!	14\!\cdot\!  64\!&\!
63\!\cdot\!  14\!&\!
aba\!\cdot\!  63\!&\!
64\!\cdot\!  aba\!&\!
54\!\cdot\!  51\!&\!
43\!\cdot\!  41\!&\!
62\!\cdot\!  13\!&\!
12\!\cdot\!  61\!&\!
61\!\cdot\!  24\!&\!
22\!\cdot\!  12\!&\!
53\!\cdot\!  53\!\!\\ \hline

\!\!41\!\!&\!	63\!\cdot\!  53\!&\!
51\!\cdot\!  63\!&\!
aba\!\cdot\!  51\!&\!
53\!\cdot\!  aba\!&\!
43\!\cdot\!  42\!&\!
31\!\cdot\!  32\!&\!
54\!\cdot\!  61\!&\!
64\!\cdot\!  52\!&\!
52\!\cdot\!  13\!&\!
14\!\cdot\!  54\!&\!
41\!\cdot\!  41\!\!\\ \hline

\!\!32\!\!&\!	51\!\cdot\!  41\!&\!
42\!\cdot\!  51\!&\!
aba\!\cdot\!  42\!&\!
41\!\cdot\!  aba\!&\!
31\!\cdot\!  34\!&\!
22\!\cdot\!  24\!&\!
43\!\cdot\!  52\!&\!
53\!\cdot\!  44\!&\!
44\!\cdot\!  61\!&\!
63\!\cdot\!  53\!&\!
32\!\cdot\!  32\!\!\\ \hline

\!\!24\!\!&\!	42\!\cdot\!  32\!&\!
34\!\cdot\!  42\!&\!
aba\!\cdot\!  34\!&\!
32\!\cdot\!  aba\!&\!
22\!\cdot\!  23\!&\!
14\!\cdot\!  13\!&\!
31\!\cdot\!  44\!&\!
41\!\cdot\!  33\!&\!
33\!\cdot\!  52\!&\!
51\!\cdot\!  41\!&\!
24\!\cdot\!  24\!\!\\ \hline

\!\!13\!\!&\!	34\!\cdot\!  24\!&\!
23\!\cdot\!  34\!&\!
aba\!\cdot\!  23\!&\!
24\!\cdot\!  aba\!&\!
14\!\cdot\!  11\!&\!
63\!\cdot\!  61\!&\!
22\!\cdot\!  33\!&\!
32\!\cdot\!  21\!&\!
21\!\cdot\!  44\!&\!
42\!\cdot\!  32\!&\!
13\!\cdot\!  13\!\!\\ \hline

\!\!61\!\!&\!	23\!\cdot\!  13\!&\!
11\!\cdot\!  23\!&\!
aba\!\cdot\!  11\!&\!
13\!\cdot\!  aba\!&\!
63\!\cdot\!  62\!&\!
51\!\cdot\!  52\!&\!
14\!\cdot\!  21\!&\!
24\!\cdot\!  12\!&\!
12\!\cdot\!  33\!&\!
34\!\cdot\!  14\!&\!
61\!\cdot\!  61\!\!\\ \hline

\!\!51\!\!&\!	13\!\cdot\!  63\!&\!
61\!\cdot\!  13\!&\!
aba\!\cdot\!  61\!&\!
63\!\cdot\!  aba\!&\!
53\!\cdot\!  52\!&\!
41\!\cdot\!  42\!&\!
64\!\cdot\!  11\!&\!
14\!\cdot\!  62\!&\!
62\!\cdot\!  23\!&\!
24\!\cdot\!  64\!&\!
51\!\cdot\!  51\!\!\\ \hline

\!\!31\!\!&\!	53\!\cdot\!  43\!&\!
41\!\cdot\!  53\!&\!
aba\!\cdot\!  41\!&\!
43\!\cdot\!  aba\!&\!
33\!\cdot\!  32\!&\!
21\!\cdot\!  22\!&\!
44\!\cdot\!  51\!&\!
54\!\cdot\!  42\!&\!
42\!\cdot\!  63\!&\!
64\!\cdot\!  44\!&\!
31\!\cdot\!  31\!\!\\ \hline

\!\!11\!\!&\!	33\!\cdot\!  23\!&\!
21\!\cdot\!  33\!&\!
aba\!\cdot\!  21\!&\!
23\!\cdot\!  aba\!&\!
13\!\cdot\!  12\!&\!
61\!\cdot\!  62\!&\!
24\!\cdot\!  31\!&\!
34\!\cdot\!  22\!&\!
22\!\cdot\!  43\!&\!
44\!\cdot\!  24\!&\!
11\!\cdot\!  11\!\!\\ \hline

\!\!62\!\!&\!	21\!\cdot\!  11\!&\!
12\!\cdot\!  21\!&\!
aba\!\cdot\!  12\!&\!
11\!\cdot\!  aba\!&\!
61\!\cdot\!  64\!&\!
52\!\cdot\!  54\!&\!
13\!\cdot\!  22\!&\!
23\!\cdot\!  14\!&\!
14\!\cdot\!  31\!&\!
33\!\cdot\!  23\!&\!
62\!\cdot\!  62\!\!\\ \hline

\!\!54\!\!&\!	12\!\cdot\!  62\!&\!
64\!\cdot\!  12\!&\!
aba\!\cdot\!  64\!&\!
62\!\cdot\!  aba\!&\!
52\!\cdot\!  53\!&\!
44\!\cdot\!  43\!&\!
61\!\cdot\!  14\!&\!
11\!\cdot\!  63\!&\!
63\!\cdot\!  22\!&\!
21\!\cdot\!  11\!&\!
54\!\cdot\!  54\!\!\\ \hline

\!\!43\!\!&\!	64\!\cdot\!  54\!&\!
53\!\cdot\!  64\!&\!
aba\!\cdot\!  53\!&\!
54\!\cdot\!  aba\!&\!
44\!\cdot\!  41\!&\!
33\!\cdot\!  31\!&\!
52\!\cdot\!  63\!&\!
62\!\cdot\!  51\!&\!
51\!\cdot\!  14\!&\!
12\!\cdot\!  62\!&\!
43\!\cdot\!  43\!\!\\ \hline

\!\!34\!\!&\!	52\!\cdot\!  42\!&\!
44\!\cdot\!  52\!&\!
aba\!\cdot\!  44\!&\!
42\!\cdot\!  aba\!&\!
32\!\cdot\!  33\!&\!
24\!\cdot\!  23\!&\!
41\!\cdot\!  54\!&\!
51\!\cdot\!  43\!&\!
43\!\cdot\!  62\!&\!
61\!\cdot\!  51\!&\!
34\!\cdot\!  34\!\!\\ \hline

\!\!23\!\!&\!	44\!\cdot\!  34\!&\!
33\!\cdot\!  44\!&\!
aba\!\cdot\!  33\!&\!
34\!\cdot\!  aba\!&\!
24\!\cdot\!  21\!&\!
13\!\cdot\!  11\!&\!
32\!\cdot\!  43\!&\!
42\!\cdot\!  31\!&\!
31\!\cdot\!  54\!&\!
52\!\cdot\!  42\!&\!
23\!\cdot\!  23\!\!\\ \hline

\!\!14\!\!&\!	32\!\cdot\!  22\!&\!
24\!\cdot\!  32\!&\!
aba\!\cdot\!  24\!&\!
22\!\cdot\!  aba\!&\!
12\!\cdot\!  13\!&\!
64\!\cdot\!  63\!&\!
21\!\cdot\!  34\!&\!
31\!\cdot\!  23\!&\!
23\!\cdot\!  42\!&\!
41\!\cdot\!  31\!&\!
14\!\cdot\!  14\!\!\\ \hline

\!\!63\!\!&\!	24\!\cdot\!  14\!&\!
13\!\cdot\!   24\!&\!
aba\!\cdot\!  13\!&\!
14\!\cdot\!   aba\!&\!
64\!\cdot\!   61\!&\!
53\!\cdot\!   51\!&\!
12\!\cdot\!   23\!&\!
22\!\cdot\!   11\!&\!
11\!\cdot\!   34\!&\!
32\!\cdot\!   22\!&\!
63\!\cdot\!  63\!\!\\ \hline

\!\!42\!\!&\!	61\!\cdot\!   51\!&\!
52\!\cdot\!   61\!&\!
aba\!\cdot\!   52\!&\!
51\!\cdot\!   aba\!&\!
41\!\cdot\!   44\!&\!
32\!\cdot\!   34\!&\!
53\!\cdot\!   62\!&\!
63\!\cdot\!   54\!&\!
54\!\cdot\!    11\!&\!
13\!\cdot\!  63\!&\!
42\!\cdot\!  42\!\!\\ \hline

\!\!22\!\!&\!	41\!\cdot\!   31\!&\!
32\!\cdot\!  41\!&\!
aba\!\cdot\!  32\!&\!
31\!\cdot\!  aba\!&\!
21\!\cdot\!  24\!&\!
12\!\cdot\!  14\!&\!
33\!\cdot\!  42\!&\!
43\!\cdot\!  34\!&\!
34\!\cdot\!  51\!&\!
53\!\cdot\!  53\!&\!
22\!\cdot\!  22
  \!\!\\ \hline
\end{array}$}

\medskip
\centerline{\small Table 7.}

\medskip

Assuming that $Q6$ is quadrati\-cal, using the properties of a quadratical quasigroup we will prove that all the remaining possible values of $x$ lead to a contradiction.

When we use a particular value of an element we will refer to the column in which this value appears in Table 7. For example, we will use the fact that $52=63\cdot 12$, from $(21)$, henceforth without mention

By $(21)$, if $52=12\cdot 53=63\cdot 12$, then $12=53\cdot 63$, and, multiplying on the right by $aba$ gives $64=41\cdot 51$,  which, along with $51\cdot 41=24$, (from $(24)$) gives $51=64\cdot 24$. This contradicts $51=64\cdot 11$, from $(21)$.

If $52=12\cdot 51=63\cdot 12$ then $12=51\cdot 63=62\cdot 64$, from $(20)$. Hence, by $(19)$ and $(20)$, $61=63\cdot 62=64\cdot 51=51\cdot 52$. Therefore, using $(24)$, $51=52\cdot 64=24\cdot 64$, a contradiction. 

If $52=12\cdot 43=63\cdot 12$  then, by $(23)$, $12=43\cdot 63=24\cdot 41$. By Proposition \ref{P3.11} we have $63\cdot 24=41\cdot 43=aba=23\cdot 24$, contradiction. 

 If $52=12\cdot 42=63\cdot 12$ then, by $(23)$, is $12=42\cdot 63=24\cdot 41$. By Proposition \ref{P3.11} and $(24)$, $51=41\cdot 42=63\cdot 24=24\cdot 64$. So, by $(20)$, $24=64\cdot 63=14$, contradiction.  

If $52=12\cdot 41=63\cdot 12$  then, by $(23)$, $12=41\cdot 63=24\cdot 41$ and so, using $(15)$, $41=63\cdot 24=63\cdot 53$, contradiction. 

If $52=12\cdot 34=63\cdot 12$  then, by $(21)$, $12=34\cdot 63=23\cdot 32$ and, by Proposition \ref{P3.11} and $(22)$, $42=32\cdot 34=63\cdot 23=63\cdot 54$, contradiction. 

If $52=12\cdot 33=63\cdot 12$  then, by $(21)$, $12=33\cdot 63=23\cdot 32$ and so, by Propositions \ref{P3.11} and \ref{P3.5}, $34=32\cdot 33=63\cdot 23=24\cdot 23$, contradiction.  

If $52=12\cdot 32=63\cdot 12$ then, by $(21)$, $12=32\cdot 63=23\cdot 32$  and so, by $(24)$,  $32=63\cdot 33=63\cdot 53$, contradiction.

If $52=12\cdot 31=63\cdot 12$ then, by $(15)$, $12=31\cdot 63=31\cdot 21$, contradiction.

If $52=12\cdot 24=63\cdot 12$ then, by $(15)$, $12=24\cdot 63=24\cdot 41$, contradiction.

If $52=12\cdot 23=63\cdot 12$ then, by $(21)$, $12=23\cdot 63=23\cdot 32$, contradiction.

If $52=12\cdot 22=63\cdot 12$ then, by $(26)$, $12=22\cdot 63=22\cdot 31$, contradiction.

If $52=12\cdot 14=63\cdot 12$  then, by Proposition \ref{P3.11},  $52=12\cdot 14=22$, contradiction.

\smallskip 

In this way we have proved that when $aba\cdot a=61$, there is no right solvability, a contradiction.

 The proof that there is no right solvability in Case $2$ ($aba\cdot a=62$), Case 3 ($aba\cdot a=63$) and Case 4 ($aba\cdot a=64$) are similar, where the values in Table $7$ are different, according to Theorem \ref{T5.2}. We omit these detailed calculations. 
\end{proof} 

There are $32$ quadratical quasigroups of order $25$ (cf. \cite{DM}). Some of them are isomorphic to quasigroups $Q1\!\times\! Q1$, $Q1\!\times\! (Q1)^*,$ $(Q1)^*\!\times\! Q1$, $(Q1)^*\!\times\! (Q1)^*$.
\begin{theorem}\label{T7.2}
Quadratical quasigroups induced by $\mathbb{Z}_{25}$ are not isomorphic to $Q1\!\times\!Q1$, $Q1\!\times\!(Q1)^*,$ $(Q1)^*\!\times\! Q1$, $(Q1)^*\!\times\!(Q1)^*$.
\end{theorem}
\begin{proof}
There are only two quadratical quasigroups induced by $\mathbb{Z}_{25}$ (cf. \cite{DM}). Their ope\-rations are given by $x\cdot y=22x+4y({\rm mod}\,25)$ and $x\circ y=4x+22y({\rm mod}\,25)$. Quasigroups $Q1$ and $(Q1)^*$ are isomorphic, respectively, to quasigroups $(\mathbb{Z}_5,\cdot)$ and $(\mathbb{Z}_5,\circ)$, where $x\cdot y=4x+2y({\rm mod}\,5)$ and $x\circ y=2x+4y({\rm mod}\,5)$. 

Suppose that $(\mathbb{Z}_{25},\cdot)$ is isomorphic to $Q1\!\times\! Q1$ or to $Q1\!\times\!(Q1)^*$. Since in $(\mathbb{Z}_5,\cdot)$ we have $x\cdot xy=yx$, in $Q1\!\times\! Q1$ and $Q1\!\times\!(Q1)^*$ for all $\overline{x}=(x,a)\ne\overline{y}=(y,a)$, $\bar{x}\cdot\bar{x}\bar{y}=\bar{y}\bar{x}$. But in $(\mathbb{Z}_{25},\cdot)$ we have $22\bar{y}+4\bar{x}=\bar{y}\bar{x}=\bar{x}\cdot\bar{x}\bar{y}= 10\bar{x}+16\bar{y}$, which implies $\bar{x}=\bar{y}$. So, $(\mathbb{Z}_{25},\cdot)$ cannot be isomorphic to $Q1\!\times\!Q1$ or $Q1\!\times\!(Q1)^*.$

In $(Q1)^*\!\times\!Q1$ and $(Q1)^*\!\times\!(Q1)^*$ for all $\overline{x}=(x,a)\ne\overline{y}=(y,a)$, we have $\bar{y}\bar{x}\cdot\bar{x}=\bar{x}\bar{y}$. But in $(\mathbb{Z}_{25},\cdot)$ we have $22\bar{x}+4\bar{y}=\bar{x}\bar{y}=\bar{y}\bar{x}\cdot\bar{x}= 9\bar{y}+17\bar{x}$, which implies $\bar{x}=\bar{y}$. So, $(\mathbb{Z}_{25},\cdot)$ also cannot be isomorphic to $(Q1)^*\!\times\!Q1$ or $(Q1)^*\!\times\!(Q1)^*$.

In the same manner we can prove that $(\mathbb{Z}_{25},\circ)$ is not isomorphic to
 $Q1\!\times\! Q1$, $Q1\!\times\! (Q1)^*,$ $(Q1)^*\!\times\! Q1$, $(Q1)^*\!\times\! (Q1)^*$.
\end{proof}

\section*{\centerline{8. Translatable groupoids}}\setcounter{section}{8}
\setcounter{theorem}{0}

Patterns of {\it translatability} can be hidden in the Cayley tables of quadratical quasigroups. One can assume the properties of quadratical quasigroups and then calculate whether translatable groupoids of various orders exist with these properties. We proceed to prove that the quadratical quasigroups   $Q1$, $(Q1)^*$, $Q3$, $(Q3)^*$, $Q4$ and $(Q4)^*$ are translatable and that $Q2$ is not translatable.

\begin{definition}\label{D8.1} A finite groupoid $Q=\{1,2,\ldots,n\}$ is called {\it $k$-translatable}, where $1\leqslant k< n$, if its Cayley table is obtained by the following rule: If the first row of the Cayley table is $a_1,a_2,\ldots,a_n$, then the $q$-th row is obtained from the $(q-1)$-st row by taking the last $k$ entries in the $(q-1)-$st row and inserting them as the first $k$ entries of the $q$-th row and by taking the first $n-k$ entries of the $(q-1)$-st row and inserting them as the last $n-k$ entries of the $q$-th row, where $q\in\{2,3,\ldots,n\}$. Then the (ordered) sequence $a_1,a_2,\ldots,a_n$ is called a {\it $k$-translatable sequence} of $Q$ with respect to the ordering $1,2,\ldots,n$. A groupoid is called a {\it translatable groupoid} if it has a $k$-translatable sequence for some $k\in\{1,2,\ldots,n\}$ . 
\end{definition}
 							
It is important to note that a $k$-translatable sequence of a groupoid $Q$ depends on the ordering of the elements in the Cayley table of $Q$. A groupoid may be $k$-translatable for one ordering but not for another (see Example \ref{Ex8.13} below). Unless otherwise stated we will assume that the ordering of the Cayley table is $1,2,\ldots,n$ and the first row of the table is $a_1,a_2,\ldots,a_n$.

\begin{proposition}\label{P8.2} The additive group $\mathbb{Z}_n$ is $(n-1)$-translatable.
\end{proposition}

The example below shows that there are $(n-1)$-translatable quasigroups of order $n$ which are not a cyclic group.

\begin{example}\label{Ex8.3} Consider the following three groupoids of order $n = 5$.
 {\small
$$
\begin{array}{lcccr}
\arraycolsep=1.5mm   \arraycolsep=1.2mm
\begin{array}{c|cccccc}						
\cdot&1&	2&	3&	4&	5\\ \hline
1&	1&	4&	2&	5&	3\\
2&	4&	2&	5&	3&	1\\
3&	2&	5&	3&	1&	4\\
4&	5&	3&	1&	4&	2\\
5&	3&	1&	4&	2&	5
\end{array}&
\ \ \ \ \ & \arraycolsep=1.2mm
\begin{array}{c|cccccc}
\cdot&	1&	2&	3&	4&	5\\ \hline
1&	2&	1&	3&	4&	5\\
2&	1&	3&	4&	5&	2\\
3&	3&	4&	5&	2&	1\\
4&	4&	5&	2&	1&	3\\
5&	5&	2&	1&	3&	4\end{array}&
\ \ \ \ \ &\arraycolsep=1.2mm
\begin{array}{c|ccccc}
\cdot&1&2&3&4&5\\ \hline
1&3&1&5&2&4\\
2&1&5&2&4&3\\
3&5&2&4&3&1\\
4&2&4&3&1&5\\
5&4&3&1&5&2
\end{array}
\end{array}
$$}

These groupoids are $4$-translatable quasigroups but they are not groups.
The first is idempotent, the se\-cond is without idempotents, the third is a cyclic quasi\-group gene\-ra\-ted by $1$ or by $5$.
\end{example}

\begin{proposition}\label{P8.4}
Any $(n-1)$-translatable groupoid of order $n$ is commutative.
\end{proposition}
\begin{proof}
In a $k$-translatable groupoid $i\cdot j=a_{(i-1)(n-k)+j}$, where the subscript is calculated modulo $n$. If $k=n-1$, then $i\cdot j=a_{i+j-1}=j\cdot i$. 
\end{proof}

\begin{theorem}\label{T8.5}
There are no $(m-1)$-translatable quadratical quasigroups of order $m$.
\end{theorem}
\begin{proof}
By Proposition \ref{P8.4} such a quasigroup is commutative. Since it also is bookend and idempotent, $x=(y\cdot x)\cdot (x\cdot y)=(x\cdot y)\cdot (x\cdot y)=x\cdot y$, so it cannot be a quasigroup.
\end{proof}

The following proposition is obvious.
\begin{proposition}\label{P8.6}
Every $1$-translatable groupoid is unipotent, i.e., in such groupoid there exists an element $a$ such that $x^2=a$ for every $x$.
\end{proposition}

\begin{corollary}\label{C8.7}
There is no idempotent $1$-translatable groupoid of order $n>1$.
\end{corollary}

\begin{proposition}\label{P8.8}
A $k$-translatable groupoid of order $n$ containing a cancellable element is a quasigroup if and only if $(k,n)=1$.
\end{proposition}
\begin{proof}
Let $Q$ be a $k$-translatable groupoid of order $n$ and let $a$ be its cancellable ele\-ment. Then in the Cayley table $[x_{ij}]_{n\times n}$ corresponding to this groupoid the $a$-row contains all elements of $Q$. Without loss of generality we can assume that this is the first row. If this row has the form $a_1,a_2,\ldots,a_n$, then other entries have the form
$x_{ij}=a_{(i-1)(n-k)+j}$,
where the subscript $(i-1)(n-k)+j$ is calculated modulo $n$. Obviously, for fixed $i=1,2,\ldots,n$, all entries $x_{i1},x_{i2},\ldots,x_{in}$ are different. 

If $(n,k)=1$, then also $(n,n-k)=1$. So, in this case, also all $x_{1j},x_{2j},\ldots,x_{nj}$ are different. Hence, this table determines a quasigroup. 

If $(n,k)=t>1$, then $(n,n-k)=t$ and the equation $(i-1)(n-k)=0$ has at least two solutions in the set $\{1,2,\ldots,n\}$. Thus, in the Cayley table of such groupoid at least two rows are identical. Hence such groupoid cannot be a quasigroup.
\end{proof}

\begin{theorem}\label{T8.9}
For every odd $n$ and every $k>1$ such that $(k,n)=1$ there is at most one idempotent $k$-translatable quasigroup. For even $n$ there are no such quasigroups.
\end{theorem}
\begin{proof}
Let $a_1,a_2,a_3,\ldots,a_n$ be the first row of a $k$-translatable quasigroup $Q$. 

This quasigroup is idempotent only in the case when in its Cayley table we have $1=x_{11}$, $2=x_{22}=a_{(n-k)+2}$, $3=x_{33}=a_{2(n-k)+3}$, $4=x_{44}=a_{3(n-k)+4}$, and so on. This means that the main diagonal of the table $[x_{ij}]_{n\times n}$ should contains ele\-ments $a_1,a_{(n-k)+2},a_{2(n-k)+3},\ldots,a_{(n-1)(n-k)+n}$, where all sub\-scripts are calculated modulo $n$. Obviously, $a_{t(n-k)+t}=a_{t'(n-k)+t'}$ only in the case when $t-tk\equiv t'-t'k({\rm mod}\,n)$, i.e., $(t-t')(k-1)\equiv 0({\rm mod}\,n)$. 
If $n$ is odd and $(n,k)=1$, then for some $k$ also is possible $(n,k-1)=1$. In this case the equation $z(k-1)\equiv 0({\rm mod}\,n)$ has only one solution $z=0$, so $t=t'$. Hence the diagonal of the table $[x_{ij}]_{n\times n}$ contains $n$ different elements.

If $n$ is even and $(n,k)=1$, then $k$ is odd. Thus, $k-1$ is even and $(n,k-1)\ne 1$. Hence, the equation $z(k-1)\equiv 0({\rm mod}\,n)$ has at least two solutions. Consequently, the diagonal of the table $[x_{ij}]_{n\times n}$ contains at least two equal elements. This contradicts to the fact that this quasigroup is idempotent. Therefore, for even $n$ there are no idempotent $k$-translatable quasigroups.
\end{proof}

\begin{corollary}\label{C8.10}
For every odd $n$ and every $k>1$ such that $(n,k)=(n,k-1)=1$ there is exactly one idempotent $k$-translatable quasigroup of order $n$.
\end{corollary}

\begin{corollary}\label{C8.11}
The first row of an idempotent $k$-translatable quasigroup $Q=\{1,2,\ldots,n\}$ has the form
$1,a_2,a_3,\ldots,a_n$, where $a_{(i-1)(n-k)+i({\rm mod}\,n)}=i\,$ for every $i\in Q$.
\end{corollary}

\begin{example}\label{Ex8.12}
Consider an idempotent quasigroup $Q=\{1,2,\ldots,7\}$. From the proof of Theorem \ref{T8.9} it follows that 
if this quasigroup is $3$-translatable, then the first row of its Cayley table has the form $1,4,7,3,6,2,5$. If it is $4$-translatable, then the first row has the form $1,3,5,7,2,4,6$.
\end{example}

\begin{example}\label{Ex8.13} The following example shows that for $Q1=\{a,ab,ba,b,aba\}$ the sequence $a, ba, aba, ab, b$ is $3$-translatable, but $Q1$ presented in the form $Q1'=\{a,b,ab,ba,aba\}$ has no translatable sequences.
 
\bigskip
{\small $
\begin{array}{|c|c|c|c|c|c|c|c|c|c|c|c|c|c|c|c|c|c|c|}\hline\rule{0mm}{3mm}
\!\!Q1\!\!\!&\!\!a\!\!\!&\!\!ab\!\!\!&\!\!ba\!\!\!&\!\!b\!\!\!&\!aba\!\!\! \!\\ \hline
a&a&ba&aba&ab&b\\	\hline 
ab&aba&ab&b&a&ba\\ \hline
ba&b&a&ba&aba&ab\\ \hline
b&ba&aba&ab&b&a\\ \hline
aba&ab&b&a&ba&aba\\ \hline
\end{array}
\ \ \ \ \ \ \ \ \ \ \
\begin{array}{|c|c|c|c|c|c|c|c|c|c|c|c|c|c|c|c|c|c|c|}\hline\rule{0mm}{3mm}
\!\!Q1'\!\!\!&\!\!a\!\!\!&\!\!b\!\!\!&\!\!ab\!\!\!&\!\!ba\!\!\!&\!\!aba\!\!\!\\ \hline
a&a&ab&ba&aba&b\\	\hline 
b&ba&b&aba&ab&a\\ \hline
ab&aba&a&ab&b&ba\\ \hline
ba&b&aba&a&ba&ab\\ \hline
aba&ab&ba&b&a&aba\\ \hline
\end{array}$}

\bigskip\noindent
The sequence $a,aba,b,a*b,b*a$ is $2$-translatable for $(Q1)^*=\{a,b*a,a*b,b,aba\}.$
$(Q1')^*=\{a,b,b*a,a*b,aba\}$ has no translatable sequence.

 {\small $$
\begin{array}{lcr}
\begin{array}{|c|c|c|c|c|c|c|c|c|c|c|c|c|c|c|c|c|c|c|}\hline\rule{0mm}{3mm}
\!\!\!(Q1)^*\!\!\!&\!\!a\!\!\!&\!\!b*a\!\!\!&\!\!a*b\!\!\!&\!\!b\!\!\!&\!\!aba\!\!\\ \hline
a&a&aba&b&\!a*b\!&\!b*a\! \\ \hline
\!b*a\!&\!a*b\!&\!b*a\!&a&aba&b\\ \hline
\!a*b\!&aba&b&\!a*b\!&\!b*a\!&a\\ \hline
b&\!b*a\!&a&aba&b&\!a*b\!\\ \hline
aba&b&\!a*b\!&\!b*a\!&a&aba\\ \hline
\end{array}
&&
\begin{array}{|c|c|c|c|c|c|c|c|c|c|c|c|c|c|c|c|c|c|c|}\hline\rule{0mm}{3mm}
(Q1)^*\!\!\!&\!\!\!a\!\!\!\!&\!\!b\!\!\!&\!\!b*a\!\!\!&\!\!a*b\!\!\!&\!\!aba\!\!\!\\ \hline
a&a&\!a*b\!&aba&b&b*a\\	\hline 
b&b*a&b&a&aba&a*b\\ \hline
b*a&a*b&aba&b*a&a&b\\ \hline
a*b&aba&b*a&b&a*b&a\\ \hline
aba&b&a&a*b&b*a&aba\\ \hline
\end{array}
\end{array}
$$}
\indent
By Corollary \ref{C8.10}, the quasigroup $Q1$ is isomorphic to a $3$-translatable quasigroup $(\mathbb{Z}_5,\circ)$ with the operation $x\circ y=4x+2y({\rm mod}\,5)$. The dual quasigroup $(Q1)^*$ is isomorphic to a $2$-translatable quasigroup $(\mathbb{Z}_5,\diamond)$ with the operation $x\diamond y=2x+4y({\rm mod}\,5)$.
\end{example}

\begin{theorem}\label{T8.14}
A groupoid isomorphic to a $k$-translatable groupoid also has a $k$-translatable sequence.
\end{theorem}
\begin{proof}
Let $\alpha$ be an isomorphism from a $k$-translatable groupoid $(Q,\cdot)$ to a groupoid $(S,\circ)$. If $Q$ is with ordering $1,2,\ldots,n$, then on $S$ we consider ordering induced by $\alpha$, namely $\alpha(1),\alpha(2),\ldots,\alpha(n)$. Suppose that the first row of the Cayley table of $Q$ has the form $a_1,a_2,\ldots,a_n$. Then in the $i$-th row and $j$-th column of this table is $x_{ij}=a_{(i-1)(n-k)+j({\rm mod}\,n)}$. Consequently, in the $\alpha(i)$-row and $\alpha(j)$-th column of the Cayley table $[z_{ij}]$ of $S$ we have $z_{\alpha(i),\alpha(j)}=\alpha(i)\circ\alpha(j)=\alpha(i\cdot j)=\alpha(x_{ij})$. Since $Q$ is $k$-translatable, for every $1\leqslant t\leqslant k$, we have $a_{i,n-k+t}=a_{i+1,t}$. Thus,
$z_{\alpha(i),\alpha(n-k+t)}=\alpha(i)\circ\alpha(n-k+t)=\alpha(x_{i,n-k+t})=\alpha(x_{i+1,t})=
\alpha((i+1)\cdot t)=\alpha(i+1)\circ\alpha(t)=z_{\alpha(i+1),\alpha(t)}$. This shows that $S$ also is $k$-translatable (for ordering $\alpha(1),\alpha(2),\ldots,\alpha(n)$).
\end{proof}
\begin{theorem}\label{T8.15} An idempotent cancellable groupoid of order $9$ is not translatable.
\end{theorem}
\begin{proof} Let $a_1,a_2,a_3,a_4,a_5,a_6,a_7,a_8,a_9$ be the first row of the Cayley table of an idempotent cancellable groupoid $Q$. Then obviously $a_i\ne a_j$ for $i\ne j$. If $Q$ is $k$-translatable, then $x_{44}=4=a_{3(9-k)+4}$. Since $3(9-k)+4\equiv 4({\rm mod}\,9)$ only for $k=3$ and $k=6$, this groupoid can be $3$-translatable or $6$-translatable. But in this case the fourth row coincides with the first, so $Q$ cannot be cancellable.
 \end{proof}

\begin{corollary}\label{C8.16} The quadratical quasigroups of order $9$ are not translatable. 
\end{corollary}

\begin{theorem}\label{T8.17}
An idempotent, bookend quasigroup $Q$, where $Q=\{1,2,\ldots,n\}$, is $k$-translatable if and only if for every $i\in Q$ we have $i=a_{(s-1)(n-k)+t({\rm mod}\,n)}$, where $s,t\in Q$ are such that
\begin{equation}\label{e14}
\arraycolsep=.5mm\rule{30mm}{0mm}
\left\{\begin{array}{rcll}k-2&\equiv&s(k-1)({\rm mod}\,n),\\
ik-1&\equiv &t(k-1)({\rm mod}\,n).
\end{array}\right.
\end{equation}
\end{theorem}
\begin{proof}
Let $1,a_2,a_3,\ldots,a_n$ be the first row of the Cayley table $[x_{ij}]$ of an idempotent, bookend quasigroup $Q=\{1,2,3,\ldots,n\}$. If it is $k$-translatable, then, by Corollary \ref{C8.11}, we have $a_{(i-1)(n-k)+i({\rm mod}\,n)}=i$ for each $i\in Q$.

Moreover, in this quasigroup for every $i\in Q$ should be 
$$
\arraycolsep=.5mm\begin{array}{rl}
i&=(1\cdot i)\cdot (i\cdot 1)=a_i\cdot x_{i1}=a_i\cdot a_{(i-1)(n-k)+1({\rm mod}\,n)}\\[3pt]
&=s\cdot t=x_{st}=a_{(s-1)(n-k)+t({\rm mod}\,n)},
\end{array}
$$
where
$$
\left\{\begin{array}{ll}
a_i=a_{(s-1)(n-k)+s({\rm mod}\,n)}=s,\\
a_{(i-1)(n-k)+1({\rm mod}\,n)}=a_{(t-1)(n-k)+t({\rm mod}\,n)}=t
\end{array}\right.
$$
for some $s,t\in \{1,2,\ldots,n\}$ satisfying \eqref{e14}.
 
The converse statement is obvious.
\end{proof}

\begin{corollary}
A quadratical quasigroup of order $25$ can be $k$-translatable only for $k=7$ or $k=18$.
\end{corollary}
\begin{proof}
Let $Q=\{1,2,\ldots,25\}$ be a quadratical quasigroup.
By Theorem \ref{T8.17}, in this quasigroup for $i=2$ should be 
$$
a_{27-k({\rm mod}\,25)}=x_{st}=a_{(s-1)(25-k)+t({\rm mod}\,25)},
$$
where $s,t\in\{1,2,\ldots,25\}$ satisfy the equations 
$$\arraycolsep=.5mm
\left\{\begin{array}{rcll}k-2&\equiv&s(k-1)({\rm mod}\,25),\\
2k-1&\equiv &t(k-1)({\rm mod}\,25).
\end{array}\right.
$$
To reduce the number of solutions of these equations observe that 
$$
x_{i1}\ne 1\longleftrightarrow a_{(i-1)(25-k)+1({\rm mod}\,25)}\ne 1=a_1\longleftrightarrow (i-1)k\equiv\!\!\!\!\!/\; 0({\rm mod}\,25).
$$ 
The last, for $i=6$, is possible only for $k\ne 5,10,15,20$.

Also 
$$
x_{ii}\ne 1\longleftrightarrow a_{(i-1)(25-k)+i({\rm mod}\,25)}\ne 1=a_1\longleftrightarrow (i-1)(k-1)\equiv\!\!\!\!\!/\;  0({\rm mod}\,25),
$$ 
which for $i=6$ is possible only for $k\ne 6,11,16,21$.

Hence $Q$ cannot be $k$-translatable for $k\in\{5,6,10,11,15,16,20,21\}$. 
By Theorem \ref{T8.5} and Corollary \ref{C8.7} it also cannot be $k$-translatable for $k\in\{1,24,25\}$.

In other cases, for $i=2$, we obtain
$$
\begin{array}{|c|c|c|c|c|c|c|c|c|c|c|c|c|c|c|}\hline
k&2&3&4&7&8&9&12&13&14&17&18&19&22&23\\ \hline
s&25&13&9&5&8&4&10&3&24&15&23&19&20&18\\ \hline
t&3&15&19&23&20&24&18&25&4&12&5&9&8&10\\\hline
\!x_{st}\!&a_5&a_4&\!a_{12}\!&\!a_{20}\!&\!a_{14}\!&\!a_{22}\!&\!a_{10}\!&\!a_{24}\!&a_{7}&\!a_{24}\!&a_{9}&\!a_{17}\!&\!a_{15}\!&\!a_{19}\!\\\hline
\!\!a_{27-k}\!\!&\!a_{25}\!&\!a_{24}\!&\!a_{23}\!&\!a_{20}\!&\!a_{19}\!&\!a_{18}\!&\!a_{15}\!&\!a_{14}\!&\!a_{13}\!&\!a_{10}\!&a_{9}&a_{8}&a_{5}&a_{4}\\\hline
\end{array}
$$

Since $x_{st}=a_{27-k}$ only for $k=7$ and $k=18$, a quasigroup of order $25$ can be $k$-translatable only for $k=7$ and $k=18$.

Direct computations shows that $\mathbb{Z}_{25}$ with the operation $x\cdot y=22x+4y({\rm mod}\,25)$ is an example of a $7$-translatable quadratical quasigroup of order $25$. Its dual quasigroup is a $18$-translatable. 
\end{proof}

By changing the order of rows and columns in Tables 3, 4, 5 and 6 we obtain the following two theorems.

\begin{theorem}\label{T8.19} The sequence $11,12,33,21,31,34,24,32,13,14,13,aba,22$ is $5$-trans\-latable for $Q3=\{ 11, 14, 34, 12, 23, 24, 33, aba, 32, 21, 22, 13, 31\}$. 

The sequence $\;11^*\!,12^*\!,23^*\!,aba^*\!,22^*\!,13^*\!,14^*\!,34^*\!,24^*\!,32^*\!,33^*\!,21^*\!,31^*\;$
is $8$-trans\-latable for $(Q3)^*\!=\!\{11^*\!,14^*\!,31^*\!,13^*\!,21^*\!,22^*\!,33^*\!,aba^*\!,32^*\!,23^*\!,24^*\!,12^*\!,34^*\}.$
\end{theorem}
\begin{theorem}\label{T8.20} The sequence 

\smallskip
\centerline{$11, 12, 42, 43, 13, 14, 33, 21, 31, 44, 23, aba, 22, 41, 34, 24, 32$}

\smallskip\noindent
is $13$-translatable for 

\smallskip
\centerline{$Q4=\{11,14,23,24,43,31,41,12,33,aba,32,13,44,34,42,21,22\}$.}

\smallskip\noindent
The sequence 

\smallskip
\centerline{$11^*,12^*,34^*,24^*,32^*,44^*,23^*,aba^*,22^*,41^*,33^*,21^*,31^*,13^*,14^*,43^*,42^*$}

\smallskip\noindent
is $4$-translatable for 

\smallskip
\centerline{
$(Q4)^*\!=\! \{11^*\!, 14^*\!, 21^*\!, 22^*\! , 44^*\! , 34^*\! , 42^*\! , 13^*\! , 33^*\! , aba^*, 32^*\! , 12^* \!, 43^*\! , 31^*\! , 41^*\! , 23^*\! , 24^*\}$.}
\end{theorem}

Quasigroups $Q3$ and $(Q3)^*$ are isomorphic, respectively, to quasigroups $(\mathbb{Z}_{13},\cdot)$ and $(\mathbb{Z}_{13},\circ)$, where $x\cdot y=11x+3y({\rm mod}\,13)$ and $x\circ y=3x+11y({\rm mod}\,13)$.

Quasigroups $Q4$ and $(Q4)^*$ are isomorphic, respectively, to quasigroups $(\mathbb{Z}_{17},\cdot)$ and $(\mathbb{Z}_{17},\circ)$, where $x\cdot y=11x+7y({\rm mod}\,17)$ and $x\circ y=7x+11y({\rm mod}\,17)$.

\section*{\centerline{9. Translatable quasigroups induced by groups $\mathbb{Z}_n$}}\setcounter{section}{9}\setcounter{theorem}{0}

In this section we describe quadratical quasigroups induced by groups $\mathbb{Z}_n$. We start with some general results.

\begin{lemma}\label{L9.1}
A quasigroup of the form $x*y=ax+by+c$ induced by a group $\mathbb{Z}_n$ is $k$-translatable if and only if  $a+kb\equiv 0({\rm mod}\,n)$.
\end{lemma}
\begin{proof}The $i$-th row of the Cayley table of this quasigroup has the form
$$
a(i-1)+c,a(i-1)+b+c,a(i-1)+2b+c,\ldots,a(i-1)+(n-1)b+c,
$$
the $(i+1)$-row has the form
$$
ai+c,ai+b+c,ai+2b+c,\ldots,ai+(n-1)b+c.
$$
So, this quasigroup is $k$-translatable if and only if 
$$
ai+c=a(i-1)+(n-k)b+c({\rm mod}\,n),
$$
i.e., if and only if $a+kb\equiv 0({\rm mod}\,n)$.
\end{proof}
\begin{corollary}\label{C9.2}
A quasigroup $(\mathbb{Z}_n,\diamond)$, where $x\diamond y=ax+y+c$, is $(n-a)$-trans\-latable.
\end{corollary}

\begin{theorem}\label{T9.3}
Each quadratical quasigroup induced by group $\mathbb{Z}_m$ is $k$-translatable for some $1<k<m-1$, namely for $k$ such that $(a-1)k\equiv a({\rm mod}\,m)$. This is valid for exactly one value of $k$.
\end{theorem}
\begin{proof}
By Theorem \ref{T2.4} and Lemma \ref{L9.1} a quadratical quasigroup induced by $\mathbb{Z}_m$ is $k$-translatable if and only if there exist $k$ such that $a\equiv (1-a)k({\rm mod}\;m)$, i.e., $(a-1)k\equiv a({\rm mod}\;m)$. Since $(a-1,m)=1$, the last equation has exactly one solution in $\mathbb{Z}_m$ (cf. \cite{Vin}).
\end{proof}

\begin{theorem}\label{T9.4}
A quadratical quasigroup $(\mathbb{Z}_m,\cdot)$ with $x\cdot y=ax+(1-a)y$ is $k$-translatable if and only if its dual quasigroup $(\mathbb{Z}_m,\circ)$, where $x\circ y=(1-a)x+ay$, is $(m-k)$-translatable.
\end{theorem}
\begin{proof}
Let $(\mathbb{Z}_m,\cdot)$ be $k$-translatable, then $(a-1)k\equiv a({\rm mod}\,m)$, i.e., $k\equiv \frac{a}{a-1}({\rm mod}\;m)$. If $(\mathbb{Z}_m,\circ)$ is $t$-translatable, then $ak\equiv (a-1)({\rm mod}\,m)$, i.e., $t\equiv \frac{a-1}{a}({\rm mod}\;m)$. ($\frac{a}{a-1}$ and $\frac{a1}{a}$ are well defined in $\mathbb{Z}_m$ because $(a,m)=(a-1,m)=1$.) Thus $k+t=\frac{2a^2-2a+1}{a(a-1)}=0({\rm mod}\;m)$, by Theorem \ref{T2.4}. Hence $k+t=m$.
\end{proof}

Note that this theorem is not valid for quasigroups which are not quadratical. Indeed, a quasigroup $(\mathbb{Z}_7,\cdot)$ with $x\cdot y=4x+y({\rm mod}\,7)$ is $3$-translatable, but its dual quasigroup $(\mathbb{Z}_7,*)$, where $x*y=x+4y({\rm mod}\,7)$, is $5$-translatable.  

\begin{corollary}\label{C9.}
There are no self-dual quadratical quasigroups induced by groups $\mathbb{Z}_m$.
\end{corollary}
\medskip

Using Theorem \ref{T9.3} we can calculate all $k$-translatable quadratical quasigroups induced by groups $\mathbb{Z}_m$. For this, it is better to rewrite the condition given in Theorem \ref{T9.3} in the form $(k-1)a\equiv k({\rm mod}\,m)$. 

\medskip\noindent
{\sc $2$-translatable quadratical quasigroups}

In this case $a\equiv 2({\rm mod}\,m)$, where $a$ satisfies \eqref{e5}. So, $5\equiv 0({\rm mod}\,m)$. Thus $m=5$. Therefore there is only one $2$-translatable quadratical quasigroup induced by $\mathbb{Z}_m$. It is induced by $\mathbb{Z}_5$ and has the form $x\cdot y=2x+4y({\rm mod}\,5)$.

\medskip\noindent
{\sc $3$-translatable quadratical quasigroups}

Then $2a\equiv 3({\rm mod}\;m)$. Since \eqref{e5} can be written in the form $2a(a-1)+1=0$, we also have $3a\equiv 2({\rm mod}\;m)$. This, together with $4a\equiv 6({\rm mod}\;m)$, implies $a=4$. Hence $8\equiv 3({\rm mod}\;m)$. Thus $m=5$. Therefore there is only one $3$-translatable quadratical quasigroup induced by $\mathbb{Z}_m$. It is induced by $\mathbb{Z}_5$ and has the form $x\cdot y=4x+2y({\rm mod}\,5)$.

\medskip\noindent
{\sc $4$-translatable quadratical quasigroups}

Now $3a\equiv 4({\rm mod}\,m)$ and  $6a\equiv 8({\rm mod}\,m)$. From \eqref{e5} we obtain $6a(a-1)+3=0$, which together with the last equation gives $8a\equiv 5({\rm mod}\,m)$. This, with $9a\equiv 12({\rm mod}\,m)$, implies $a=7$. Hence $21\equiv 4({\rm mod}\,m)$. Thus $m=17$. Therefore there is only one $4$-translatable quadratical quasigroup induced by $\mathbb{Z}_m$. It is induced by $\mathbb{Z}_{17}$ and has the form $x\cdot y=7x+11y({\rm mod}\,17)$.

\medskip\noindent
{\sc $5$-translatable quadratical quasigroups}

Now $4a\equiv 5({\rm mod}\,m)$ and  $5a\equiv 3({\rm mod}\,m)$, by \eqref{e5}. Thus, $16a\equiv 20({\rm mod}\,m)$ and $15a\equiv 9({\rm mod}\,m)$, which implies $a=11$. Hence $44\equiv 5({\rm mod}\,m)$. Thus $m=13$. Therefore a $5$-translatable quadratical quasigroup is induced by $\mathbb{Z}_{13}$ and has the form $x\cdot y=11x+4y({\rm mod}\,13)$.

\medskip\noindent
{\sc $6$-translatable quadratical quasigroups}

Now $5a\equiv 6({\rm mod}\,m)$ and $12a\equiv 7({\rm mod}\,m)$, by \eqref{e5}. Thus, $25a\equiv 30({\rm mod}\,m)$ and $24a\equiv 14({\rm mod}\,m)$, which implies $a=16$. Hence $80\equiv 6({\rm mod}\,m)$. Thus $m=37$. Therefore a $6$-translatable quadratical quasigroup is induced by $\mathbb{Z}_{37}$ and has the form $x\cdot y=16x+22y({\rm mod}\,37)$.

\medskip\noindent
{\sc $7$-translatable quadratical quasigroups}

Now $6a\equiv 7({\rm mod}\,m)$ and  $7a\equiv 4({\rm mod}\,m)$, by \eqref{e5}. Thus, $a\equiv (-3)({\rm mod}\,m)$ and $(-18)\equiv 7({\rm mod}\,m)$. Consequently, $25\equiv 0({\rm mod}\,m)$. Hence $m=25$. (The case $m=5$ is impossible because must be $m>k=7$.) Therefore $a=22$. So, a $7$-translatable quadratical quasigroup is induced by $\mathbb{Z}_{25}$ and has the form $x\cdot y=22x+4y({\rm mod}\,25)$.

\medskip\noindent
{\sc $8$-translatable quadratical quasigroups}

Now $7a\equiv 8({\rm mod}\,m)$ and  $16a\equiv 9({\rm mod}\,m)$, by \eqref{e5}. Thus, $49a\equiv 56({\rm mod}\,m)$ and $48a\equiv 27({\rm mod}\,m)$ shows that $a\equiv 29({\rm mod}\,m)$. Hence $7\cdot 29\equiv 8({\rm mod}\,m)$ and $16\cdot 29\equiv 9({\rm mod}\,m)$ imply $195\equiv 0({\rm mod}\,m)$ and $455\equiv 0({\rm mod}\,m)$. Therefore, $65\equiv 0({\rm mod}\,m)$. Since $m>k=8$, the last means that $m=65$ or $m=13$. So, a $8$-translatable quadratical quasigroup is induced by $\mathbb{Z}_{13}$ or by $\mathbb{Z}_{65}$. In the first case it has the form $x\cdot y=3x+11y({\rm mod}\,13)$, in the second $x\cdot y=29x+37y({\rm mod}\,65)$.

\medskip\noindent
{\sc $9$-translatable quadratical quasigroups}

In this case $8a\equiv 9({\rm mod}\,m)$ and  $9a\equiv 5({\rm mod}\,m)$, by \eqref{e5}. So, $a\equiv (-4)({\rm mod}\,m)$, and consequently $41\equiv 0({\rm mod}\,m)$. Thus, $m=41$. Hence a $9$-translatable quadratical quasigroup is induced by $\mathbb{Z}_{41}$ and has the form $x\cdot y=37x+5y({\rm mod}\,41)$.

\medskip\noindent
{\sc $10$-translatable quadratical quasigroups}

In a similar way we can see that there is only one $10$-translatable quasigroup induced by $\mathbb{Z}_m$. It is induced by $\mathbb{Z}_{101}$ and has the form $x\cdot y=46x+56y({\rm mod}\,101)$.

\medskip
As a consequence of the above calculations and Theorem \ref{T9.4} we obtain the following list of $(m-k)$-translatable quadratical quasigroups induced by $\mathbb{Z}_m$.

\medskip\noindent
{\sc $(m\!-\!2)$-translatable quadratical quasigroups}

There is only one such quasigroup. It is induced by $\mathbb{Z}_5$ and has the form $x\cdot y=4x+2y({\rm mod}\,5)$.

\medskip\noindent
{\sc $(m\!-\!3)$-translatable quadratical quasigroups}

There is only one such quasigroup. It has the form $x\cdot y=2x+4y({\rm mod}\,5)$.

\medskip\noindent
{\sc $(m\!-\!4)$-translatable quadratical quasigroups}

There is only one such quasigroup. It has the form $x\cdot y=11x+7y({\rm mod}\,17)$.

\medskip\noindent
{\sc $(m\!-\!5)$-translatable quadratical quasigroups}

There is only one such quasigroup. It has the form $x\cdot y=3x+11y({\rm mod}\,13)$.

\medskip\noindent
{\sc $(m\!-\!6)$-translatable quadratical quasigroups}

There is only one such quasigroup. It has the form $x\cdot y=22x+16y({\rm mod}\,37)$.

\medskip\noindent
{\sc $(m\!-\!7)$-translatable quadratical quasigroups}

There is only one such quasigroup. It has the form $x\cdot y=4x+22y({\rm mod}\,25)$.

\medskip\noindent
{\sc $(m\!-\!8)$-translatable quadratical quasigroups}

There are only two such quasigroups. The first has the form $\cdot x\cdot y=11x+3y({\rm mod}\,13)$, the second $x\cdot y=37x+29y({\rm mod}\,65)$.

\medskip\noindent
{\sc $(m\!-\!9)$-translatable quadratical quasigroups}

There is only one such quasigroup. It has the form $x\cdot y=5x+37y({\rm mod}\,41)$.

\medskip\noindent
{\sc $(m\!-\!10)$-translatable quadratical quasigroups}

Such a quasigroup is induced by $\mathbb{Z}_{101}$ and has the form $x\cdot y=56x+46y({\rm mod}\,101)$.

\medskip   
Below, for $k<40$, we list all $k$-translatable quadratical quasigroups of order $m\leqslant 1200$ defined on $\mathbb{Z}_m$.
$$
\begin{array}{ccccccccc}
\begin{array}{|c|c|c|c|}\hline
k&m&a&b\\ \hline
2&5&2&4\\ \hline
3&5&4&2\\ \hline
4&17&7&11\\ \hline
5&13&11&7\\ \hline
6&37&16&22\\ \hline
7&25&22&4\\ \hline
8&13&3&11\\
&65&29&37\\ \hline
9&41&37&5\\ \hline
10&101&46&56\\ \hline
11&61&56&6\\ \hline
12&29&9&21\\ 
&145&67&79\\ \hline
13&17&11&7\\ 
&85&79&7\\ \hline
14&197&92&106\\ \hline
15&113&106&8\\ \hline
16&257&121&137\\ \hline
\end{array}
&&
\begin{array}{|c|c|c|c|}\hline
k&m&a&b\\ \hline
17&29&21&9\\ 
&145&137&9\\ \hline
18&25&4&22\\ 
&65&24&42\\ 
&325&154&172\\ \hline
19&181&172&10\\ \hline
20&401&191&211\\ \hline
21&221&211&11\\ \hline
22&97&38&60\\
&485&232&254\\ \hline
23&53&42&12\\ 
&265&254&12\\ \hline
24&577&277&301\\ \hline
25&313&301&13\\ \hline
26&677&326&352\\ \hline
27&73&60&14\\ 
&365&352&14\\ \hline
28&157&65&93\\ 
&785&379&407\\ \hline
\end{array}
&&
\begin{array}{|c|c|c|c|}\hline
k&m&a&b\\ \hline
29&421&407&15\\ \hline
30&53&12&42\\ 
&901&436&466\\ \hline
31&37&22&16\\ 
&481&466&16\\ \hline
32&41&5&37\\ 
&205&87&119\\ 
&1025&497&529\\ \hline
33&109&93&17\\ 
&545&529&17\\ \hline
34&89&28&62\\ %
&1157&562&596\\ \hline
35&613&596&18\\ \hline
36&1297&631&667\\ \hline
37&137&119&19\\ 
&685&667&198\\ \hline
38&85&24&62\\
&289&126&164\\ \hline
39&761&742&20\\ \hline
\end{array}
\end{array}
$$

\section*{\centerline{10. Classification of quadratical quasigroups}}

Below are listed all $k$-translatable quadratical quasigroups of the form $x\cdot y=ax+by({\rm mod}\,m)$, where $a<b$, defined on the group $\mathbb{Z}_m$ for $m<500$. Dual quasigroups $x\circ y=bx+ay({\rm mod}\,m)$ are omitted.

 For example, the group $\mathbb{Z}_{65}$  induces four quadratical quasigroups: $x\cdot y=24x+42y({\rm mod}\,65)$, $x\cdot y=29x+37y({\rm mod}\,65)$ and two duals to these two. The first is $18$-translatable, the second $8$-translatable. In the table below these dual quasigroups $x\cdot y=42x+24y({\rm mod}\,65)$ and $x\cdot y=37x+29y({\rm mod}\,65)$ are not listed.\\

\noindent
$
\begin{array}{cccccc}
\begin{array}{|c|c|c|c|}\hline
m&a&b&k\\ \hline
5&2&4&2\\ \hline
13&3&11&8\\ \hline
17&7&11&4\\ \hline
25&4&22&18\\ \hline
29&9&21&12\\ \hline
37&16&22&6\\ \hline
41&5&37&32\\ \hline
53&12&42&30\\ \hline
61&6&56&50\\ \hline
65&24&42&18\\ 
&29&37&8\\ \hline
73&14&60&46\\ \hline
85&7&79&72\\
&24&62&38\\ \hline
89&28&62&34\\ \hline
97&38&60&22\\ \hline
101&46&56&10\\ \hline
109&17&93&76\\ \hline
113&8&106&98\\ \hline
125&29&97&68\\ \hline
137&19&119&100\\ \hline
145&9&137&128\\ 
&67&79&12\\ \hline
149&53&97&44\\ \hline
157&65&93&28\\ \hline
169&50&120&70\\ \hline
\end{array}\rule{0mm}{15mm}

&&
\begin{array}{|c|c|c|c|}\hline
m&a&b&k\\ \hline
173&47&127&80\\ \hline
181&10&172&162\\ \hline
185&22&164&142\\
&59&127&68\\ \hline
193&41&153&112\\ \hline
197&92&106&14\\ \hline
205&37&169&132\\ 
&87&119&32\\ \hline
221&11&211&200\\
&24&198&174\\ \hline
229&54&176&122\\ \hline
233&45&189&144\\ \hline
241&89&153&64\\ \hline
257&121&137&16\\ \hline
265&12&254&242\\
&42&224&182\\ \hline
269&94&176&82\\ \hline
277&109&169&60\\ \hline
281&27&255&228\\ \hline
289&126&164&38\\ \hline
293&78&216&138\\ \hline
305&67&239&172\\ 
&117&189&72\\ \hline
313&13&301&288\\ \hline
317&102&216&114\\ \hline
325&29&297&268\\ 
&154&172&18\\ \hline

\end{array}
&&
\begin{array}{|c|c|c|c|}\hline
m&a&b&k\\ \hline
337&95&243&148\\ \hline
349&107&243&136\\ \hline
353&156&198&42\\ \hline
365&14&352&338\\
&87&279&192\\ \hline
373&135&239&104\\ \hline
377&50&328&278\\
&154&224&70\\ \hline
389&58&332&274\\ \hline
397&32&366&334 \\ \hline
401&191&211&20\\ \hline
409&72&338&266\\ \hline
421&15&407&392\\ \hline
425&79&347&268\\
&147&279&132\\ \hline
433&90&344&254\\ \hline
445&62&384&322\\
&117&329&212\\ \hline
449&34&416&382\\ \hline
457&55&403&348\\ \hline
461&207&255&48\\ \hline
481&16&466&450\\
&133&349&216\\ \hline
485&157&329&172\\ 
&232&254&22\\ \hline
493&79&415&336\\
&96&398&302\\ \hline
\end{array}
\end{array}
$


\section*{\centerline{10. Open questions and problems}}\setcounter{section}{9}
\setcounter{theorem}{0}

\noindent
{\bf Problem 1.} {\it For which values of $n$ are there quadratical quasigroups of form $Qn$?} 

Note that $n\not\in\{5,6,8,14,17,19,33,26,32,\ldots\}$. Moreover, from Theorem 4.11 in \cite{DM} it follows that there are no such quasigroups if there is a prime $p|4n+1$ such that $p\equiv 3({\rm mod}\,4)$.

\medskip\noindent
{\bf Problem 2.} {\it Is every quadratical quasigroup of form $Qn$ translatable?}

The answer is positive if $Qn$ is isomorphic to a quasigroup induced by $\mathbb{Z}_m$.

\medskip\noindent
{\bf Problem 3.} {\it Are there self-dual, quadratical groupoids of order greater than $9$?}

Such quasigroups cannot be induced by $\mathbb{Z}_m$.

\medskip\noindent
{\bf Problem 4.} {\it Is every quadratical groupoid of order greater than $9$ and of form $Qn$ $(n\geqslant 3)$  generated by any two of its distinct elements?}

\medskip\noindent
{\bf Problem 5.} {\it	If a quadratical quasigroup $Q$ of order $m$ is $k$-translatable, then is $Q^*$ $(m-k)$-translatable?}

For quadratical quasigroups induced by $\mathbb{Z}_m$ the answer is positive.

\noindent
W.A. Dudek \\
 Faculty of Pure and Applied Mathematics,
 Wroclaw University of Science and Technology,
 50-370 Wroclaw,  Poland \\
 Email: wieslaw.dudek@pwr.edu.pl\\[4pt]
R.A.R. Monzo\\
Flat 10, Albert Mansions, Crouch Hill, London N8 9RE, United Kingdom\\
E-mail: bobmonzo@talktalk.net
}

\end{document}